\documentclass[10pt]{article}
\usepackage{amsmath}
\numberwithin{equation}{section}
\usepackage{amsfonts}
\usepackage{amssymb}
\usepackage{graphicx}
\usepackage{mathrsfs}
\usepackage{xcolor}
\usepackage{verbatim}
\usepackage{mathrsfs}
\usepackage[body={15.5cm,21cm}, top=3cm]{geometry}
\usepackage{paralist}
\usepackage{ntheorem}
\usepackage{appendix}

\allowdisplaybreaks[4]
\usepackage{hyperref}
\hypersetup{colorlinks=true,
linkcolor=blue,
anchorcolor=blue,
citecolor=blue}
\providecommand{\U}[1]{\protect \rule{.1in}{.1in}}
\newtheorem{theorem}{Theorem}[section]

\newtheorem{definition}[theorem]{Definition}

\newtheorem{lemma}[theorem]{Lemma}

\newtheorem{proposition}[theorem]{Proposition}
\newtheorem{remark}[theorem]{Remark}
\newtheorem{assumption}[theorem]{Assumption}

\theoremstyle{empty}

\newenvironment{proof}[1][Proof]{\noindent \textbf{#1.} }{\  \rule{0.5em}{0.5em}}
\allowdisplaybreaks[2]
\begin{document}
\title{Quadratic BSDEs with double constraints driven by
$G$-Brownian motion}
\author{Wei He\thanks{Research Center for Mathematics 
and Interdisciplinary Sciences, Frontiers Science 
Center for Nonlinear Expectations (Ministry of Education), 
Shandong University, Qingdao, China (hew@sdu.edu.cn).
Research supported by China Postdoctoral Science Foundation
(2024M761781), the Natural Science Foundation of Shandong Province (ZR2024QA186), 
the Fundamental Research Funds for the Central Universities.}
\and
Qiangjun Tang \thanks{Zhongtai Securities Institute for Financial 
Studies, Shandong University, Jinan, China (tangqj@mail.sdu.edu.cn).}}
\date{}
\maketitle

\begin{abstract}
    In this paper, we investigate the well-posedness of 
    quadratic backward stochastic differential equations 
    driven by $G$-Brownian motion (referred to as G-BSDEs) with double mean reflections. 
    By employing a representation of the solution via $G$-BMO martingale techniques, 
    along with fixed point arguments, the Skorokhod problem, the backward Skorokhod problem, 
    and the $\theta$-method, we establish existence and uniqueness results for such $G$-BSDEs 
    under both bounded and unbounded terminal conditions. 
\end{abstract}


\textbf{MSC-classification}: 60H10, 60H30

\section{Introduction}\label{sec1}
The study of backward stochastic differential equations (BSDEs) was initiated by Pardoux and Peng \cite{PP90}, 
who established the fundamental existence and uniqueness results under Lipschitz conditions. 
Since then, BSDEs have found widespread applications in mathematical finance, stochastic control, 
and partial differential equations (PDEs). A particularly critical extension is the calss of quadratic BSDEs, 
where the generator grows quadratically in the second variable $Z$. 
Such equations naturally arise in problems involving utility optimization, risk-sensitive control, and indifference pricing.  
Early breakthroughs in this area include Kobylanski \cite{Kobylanski2000}, who first proved well-posedness 
for bounded terminal conditions, followed by Briand and Hu \cite{BH06,BH08}, who extended these results 
to unbounded terminal conditions.

Reflected BSDEs, introduced by El Karoui et al. \cite{EKPPQ97}, incorporate constraints on the 
solution process $Y$: specially, $Y_t\geq L_t$ for all $t\in[0,T]$, where $L$ is a given continuous process. 
Subsequent works (see \cite{CM08,GIOOQ17,HL00,KT12,KT13,KLQT02,PX05} and references therein) generalized 
the theory of reflected BSDEs further.
A more advanced generalization, mean-reflected BSDEs, imposes constraints in an average sense:
for a given loss function $L$, the condition $\mathbb{E}[L(t,Y_t)] \geq 0 $ must hold. 
This formulation, studied by Briand et al. \cite{BEH18}, 
is particularly useful in risk management and robust pricing, where pathwise constraints  
are too restrictive. Additionally, Falkowski and Słomiński \cite{FS22} focused on the mean-reflection with two 
constraints, i.e., $\mathbb{E}[L(t,Y_t)]\in[l_t,r_t]$.
Later, Li \cite{Li24} developed the 
theory of double mean reflections, requiring the solution to satisfy 
$\mathbb{E}[L(t,Y_t)]\leq0\leq\mathbb{E}[R(t,Y_t)]$ (where $L,R$ are two given nonlinear loss functions with $L\leq R,$)
and introduced the backward Skorokhod problem as a key tool to construct solutions.
We also paid special attention to that Li and Shi \cite{LS25} analyzed the mean field backward stochastic 
differential equations with double mean reflections (whose generator and constraints both 
depend on the distribution of the solution) under both Lipschitz and quadratic conditions. 

The framework of $G$-expectation, proposed by Peng \cite{P07a,P08a,P19}, provids a mathematical foundation 
for modeling uncertainty in volatility.
Within this framework, backward stochastic differential equations 
by $G$-Brownian motion ($G$-BSDEs) have been extensively studied.
Owing to the nonlinear structure, a typical $G$-BSDEs reads
\begin{small}
    \begin{equation*}
        Y_{t} =\xi
        +\int_{t}^{T}f(s,Y_{s},Z_{s})ds+\sum_{i,j=1}^{d}\int_{t}^{T}g^{ij}(s,Y_{s},Z_{s})d\langle
        B^i,B^j\rangle _{s} 
        -\int_{t}^{T}Z_{s}dB_{s}-(K_{T}-K_{t}).
    \end{equation*}
\end{small}Hu et al. \cite{HJPS1} first established the well-posedness of $G$-BSDEs under Lipschitz conditions. 
Since then, many works have been devoted to studying the existence and uniqueness of solutions 
under various non-Lipschitz conditions and the related applications, which can be found in 
\cite{HQW20,Liu20,WZ21,He22,Zhao24}. 
Notably, Hu et al. \cite{HLS18} studied quadratic $G$-BSDEs with bounded terminal condition, 
and later, Hu et al. \cite{HTW22} used the $\theta$-method to approximate the solution of quadratic
G-BSDEs for generators with convexity (or concavity) and terminal value of
exponential moments.
 
%while Hu et al. (2018), Hu et al. (2018) and Hu et al. (2022) extended the theory to 
%quadratic $G$-BSDEs, using $G$-BMO martingale techniques and approximation methods.  

For reflected G-BSDEs, cases involving upper obstacles, lower obstacles, and double obstacles were 
examined under Lipschitz conditions in \cite{LPS18,LP20,LS21}, respectively.  
Cao and Tang \cite{CT20} studied reflected quadratic $G$-BSDEs with the generator growing
quadratically in the second unknown by the penalty method.
Although the development of mean-reflected G-BSDEs is still in its early stages, 
substantial theoretical results have been achieved.
Liu and Wang \cite{LW19} were the first to establish the well-posedness of mean-reflected $G$-BSDEs for two
special kinds of generators under Lipschitz conditions. 
He \cite{H24} extended the results in \cite{LW19} to a multi-dimensional and time-varying
non-Lipschitz settings, covering the general Lipschitz case.

The combination of quadratic growth and mean reflections in the $G$-framework presents 
new challenges due to the interplay between nonlinear expectations and reflection constraints.
Gu et al. \cite{GuLinXu23} addressed quadratic $G$-BSDEs with mean reflection by combining $G$-BMO martingale techniques, 
fixed-point arguments, and the $\theta$-method. Recently, He and Li \cite{HL24} investigated 
mean-reflected $G$-BSDEs with two nonlinear reflecting boundaries, expressed as:
\begin{small}
    \begin{equation*}
        \begin{cases}
            Y_t=\xi+\int_t^T f(s,Y_s,Z_s)ds+\int_t^T g(s,Y_s,Z_s)d\langle B\rangle_s-\int_t^T Z_s dB_s-(K_T-K_t)+(A_T-A_t), \\
            \widehat{\mathbb{E}}[L(t,Y_t)]\leq 0\leq \widehat{\mathbb{E}}[R(t,Y_t)], \\
            A_t=A^R_t-A^L_t \textrm{ and } \int_0^T \widehat{\mathbb{E}}[R(t,Y_t)]dA_t^R=\int_0^T \widehat{\mathbb{E}}[L(t,Y_t)]dA^L_t=0,
        \end{cases}
    \end{equation*}
\end{small}where $K$ is a non-increasing $G$-martingale and loss functions $L,R$ satisfy certain regularity conditions.
The solution to this kind of constrained $G$-BSDE is a quadruple
of processes $(Y,Z,K,A)$ with $A$ being the difference between two nondecreasing deterministic functions
$A^R, A^L,$ each satisfy the Skorokhod condition. They established the well-posedness 
of the above mean reflected $G$-BSDE under Lipschitz condition and so-called $\beta$-order Mao's condition. 
However, the case of double mean reflections with quadratic generators remained unresolved. 

This paper bridges this gap by establishing the well-posedness of quadratic $G$-BSDEs 
with double mean reflections, covering both bounded and unbounded terminal conditions. 
Our methodology integrates:  
\begin{enumerate}
    \item $G$-BMO martingale techniques: to handle the quadratic growth of the generator.
    \item Fixed-point arguments and Skorokhod problem methods: to address the reflection constraints.
    \item The $\theta$-method: for unbounded terminal conditions, requiring convexity or concavity of the generator.
\end{enumerate}

For bounded terminal conditions, we observe that the non-decreasing 
deterministic process $A$ can be represented via the solution of a standard quadratic $G$-BSDE,
a technique previously employed in settings such as those studied in \cite{BEH18} and \cite{HLS18}.
This allows us to establish well-posedness under the simplified assumption that the generator does not 
depend on $Y$. Specifically, to prove uniqueness, we use a comparison argument (via Theorem 5.1 in \cite{CT20}) 
to show that the difference between two solutions leads to identical processes; the uniqueness of 
the deterministic reflection process $A$ is proved by a Skorokhod condition argument similar to Proposition 3.4 in \cite{HL24}.
For existence, we construct a deterministic process $A$ via the backward Skorokhod problem,
then define the solution as the sum of $A$ and the solution to a standard quadratic G-BSDE (without reflection),
and verify that the Skorokhod conditions are satisfied. 
We then extend this to the general case using a fixed-point argument: first, we define a mapping 
$\Gamma:S_G^\infty\to S_G^\infty$, and use Girsanov transformation under a new $G$-expectation 
$\tilde{\mathbb{E}}$ to linearize the difference between two solutions; next, we show that 
$\Gamma$ is a contraction on small time intervals (using estimates from the Skorokhod problem and 
$G$-BMO martingale theory) to obtain a local solution; finally, we carefully stitch these 
local solutions into a global solution.

For unbounded terminal conditions, we first prove uniqueness and existence on small intervals, 
then extend local well-posedness to the global case. For local uniqueness, we first define scaled differences 
between two solutions by $\theta$-method, and use exponential bounds for $G$-BSDE (Lemma \ref{lemma-priori-estimate}) 
and Doob's maximal inequality under $G$-framework to derive the exponential moment estimates; 
then, we show that for sufficiently small interval, the scaled differences converge to $0$ as $\theta\to 1$ 
(implying the two solutions coincide). To establish local existence, we first define a sequence of quadratic $G$-BSDEs 
with double mean reflections by employing the truncation method and provide a representation for the solution
to a sequence of quadratic $G$-BSDEs by backward Skorokhod problem; then, we prove that the sequence of solutions
is a Cauchy sequence when time interval is sufficiently small by applying the $\theta$-method and some uniform estimates;
next, we show the convergence of the other components by $G$-It\^{o} formula, the uniform estimates; 
finally, we verify the limit to the sequence of solutions satisfies the original quadratic $G$-BSDEs 
with double mean reflections.
%and check the Skorokhod conditions 
%holds by similar analysis as Step 4 of Theorem 4.4 in \cite{LS25}.
   
Our results unify and generalize earlier work on single mean-reflected $G$-BSDEs to double case, 
providing a more robust framework for problems under volatility ambiguity. 
It reveals that volatility uncertainty not only complicates the analysis but also introduces a 
richer structure for connecting the reflection process to quadratic $G$-BSDEs—a feature absent in classical settings.
Potential applications include robust pricing, ambiguity-sensitive risk measures, and recursive utility optimization 
in incomplete markets, particularly when traditional Lipschitz or linear-quadratic assumptions are inadequate. 
However, our reliance on convexity/concavity and exponential moment conditions may limit applicability in 
certain pathological cases, suggesting directions for future research to relax these assumptions.

The remainder of the paper is organized as follows: Section \ref{sec2} reviews preliminary concepts, 
including $G$-expectation theory, $G$-BMO martingales, and the Skorokhod problem. 
Section \ref{sec3} presents our main results, divided into the cases of bounded and unbounded terminal conditions.
Some detailed technical proofs are presented in Appendix.
%Finally, we conclude with applications and potential extensions of our work.

\section{Preliminaries}\label{sec2} 
In this section, we introduce basic concepts and foundational results
that will be required for our analysis, including $G$-expectation theory, 
$G$-BMO martingale theory 
and the backward Skorokhod problem. For more details, readers may refer to 
\cite{CT20,DHP11,HJPS1,HLS18,HTW22,Li23,Li24,P07a,P08a,P19}.

\subsection{$G$-expectation theory}
The purpose of this part is to recall some notations and results of $G$-expectation,
$G$-Brownian motion and related $G$-stochastic calculus. 
More details can be found in \cite{P07a,P08a,P19}. 
For simplicity, we only consider the $1$-dimensional $G$-Brownian motion. 
The results still hold for the multi-dimensional case.

Let $\Omega_T=C_{0}([0,T];\mathbb{R})$ denote the space of
real-valued continuous functions with $\omega_0=0$, equipped
with the supremum norm and 
let $B$ be the canonical
process. Define
\begin{small}
    \begin{align*}
        L_{ip} (\Omega_T):=\{ \varphi(B_{t_{1}},...,B_{t_{n}}):  \ n\in\mathbb {N}, \ t_{1}
        ,\cdots, t_{n}\in[0,T], \ \varphi\in C_{b,Lip}(\mathbb{R}^{ n})\},
    \end{align*}
\end{small}where $C_{b,Lip}(\mathbb{R}^{ n})$ consists of bounded 
Lipschitz functions on $\mathbb{R}^{n}$.

We fix a sublinear, continuous and monotone function  
$G:\mathbb{R}\rightarrow\mathbb{R}$ defined by
\begin{small}
    \begin{align*}
        G(a):=\frac{1}{2}(\bar{\sigma}^2a^+-\underline{\sigma}^2a^-),
    \end{align*}
\end{small}where $0\leq \underline{\sigma}^2<\bar{\sigma}^2$. 
The related $G$-expectation on $(\Omega,L_{ip}(\Omega_T))$ is 
constructed as follows: 
for $\xi\in L_{ip}(\Omega_T)$ represented as
$\xi=\varphi(B_{{t_1}}, B_{t_2},\cdots,B_{t_n})$,
the conditional $G$-expectation is 
\begin{small}
    \begin{align*}
        \widehat{\mathbb{E}}_{t}[\varphi(B_{{t_1}}, B_{t_2},\cdots,B_{t_n})]=u_k(t, B_t;B_{t_1},\cdots,B_{t_{k-1}}),
    \end{align*}
\end{small}where $u_k(t,x;x_1,\cdots,x_{k-1})$  
solves the following fully nonlinear PDE:
\begin{small}
    \begin{align*}
        \partial_t u_k+G(\partial_x^2 u_k)=0
    \end{align*}
\end{small}on $[t_{k-1},t_k)\times\mathbb{R}$ with terminal conditions
\begin{small}
    \begin{align*}
        u_k(t_k,x;x_1,\cdots,x_{k-1})=u_{k+1}(t_k,x;x_1,\cdots,x_{k-1},x), \ k<n
    \end{align*}
\end{small}and $u_n(t_n,x;x_1,\cdots,x_{n-1})=\varphi(x_1,\cdots,x_{n-1},x)$. 
The $G$-expectation of $\xi$ is $\widehat{\mathbb{E}}_0[\xi]$, where we omit the subscript $0$
for simplicity. 
The triple $(\Omega,L_{ip}(\Omega_T),\widehat{\mathbb{E}})$ is called the $G$-expectation space.

For each $p\geq1$, let $L_{G}^{p}(\Omega_T)$
be the completion of $L_{ip} (\Omega_T)$ under the norm 
$\Vert\cdot\Vert_{L_{G}^{p}}:=(\widehat{\mathbb{E}}[|\cdot|^{p}])^{1/p}$. 
The space $L_G^\infty(\Omega_T)$ is defined as the  
Banach completion of $L_{ip}(\Omega_T)$ under the norm
$\|\xi\|_{L_G^\infty}:=\inf\{M\geq 0:|\xi|\leq M,~\text{q.s.}\}.$
The conditional $G$-expectation $\mathbb{\hat{E}}_{t}[\cdot]$ can be
extended continuously to $L_{G}^{p}(\Omega_T)$.

A $G$-expectation 
has the following representation which was proved by Denis, Hu and Peng in \cite{DHP11}.
\begin{theorem}[\cite{DHP11}]\label{thm2.1}
	There exists a weakly compact set
	$\mathcal{P}$ of probability
	measures on $(\Omega_T,\mathcal{B}(\Omega_T))$, such that
    \begin{small}
        \begin{equation*}
            \widehat{\mathbb{E}}[\xi]=\sup_{P\in\mathcal{P}}E_{P}[\xi] 
            \text{ for all } \xi\in  {L}_{G}^{1}{(\Omega_T)}.
        \end{equation*}
    \end{small}$\mathcal{P}$ is called a set that represents $\widehat{\mathbb{E}}$.
\end{theorem}

Let $\mathcal{P}$ be a weakly compact set that represents 
$\widehat{\mathbb{E}}$, we define a capacity 
\begin{small}
    \begin{equation*}
        c(A):=\sup_{P\in\mathcal{P}}P(A),\quad A\in\mathcal{B}(\Omega).
    \end{equation*}
\end{small}A set $A\in\mathcal{B}(\Omega_T)$ is called polar if $c(A)=0$.  A
property holds ``quasi-surely'' (q.s.) if it holds outside a
polar set. In this paper, we do not distinguish two random 
variables $X$ and $Y$ if $X=Y$, q.s..

The conditional $G$-expectation admits a similar representation.
\begin{theorem}[\cite{STZ12}]\label{thm2.2}
    For any $\xi\in L_G^1(\Omega_T)$, $t\in[0,T]$ and $P\in \mathcal{P}$,
    \begin{small}
        \begin{equation*}
            \widehat{\mathbb{E}}_t[\xi]=\underset{P^\prime\in \mathcal{P}(t, P)}
            {ess\sup}\ E^{ P^\prime}_t[\xi],\quad P \text{-a.s.},
        \end{equation*}
    \end{small}where $\mathcal{P}(t, P):=\{ P^\prime\in \mathcal P:  P^\prime= P~\text{on}~\mathcal F_t\}.$
\end{theorem}

\begin{definition}\label{def2.1}
	Let $M_{G}^{0}(0,T)$ be the collection of processes of the form:
    \begin{small}
        \begin{equation*}
            \eta_{t}(\omega)=\sum_{j=0}^{N-1}\xi_{j}(\omega)\mathbf{1}_{[t_{j},t_{j+1})}(t),
        \end{equation*}
    \end{small}where $\xi_{i}\in L_{ip}(\Omega_{t_{i}})$, $i=0,1,2,\cdot\cdot\cdot,N-1$. 
    For each $p\geq1$ and $\eta\in M_G^0(0,T)$, define 
    $\small \|\eta\|_{H_G^p}:=\tiny \biggl\{\widehat{\mathbb{E}}[(\int_0^T|\eta_s|^2ds)^{p/2}]\biggr\}^{1/p}$, 
    $\small \Vert\eta\Vert_{M_{G}^{p}}:=\tiny \biggl(\widehat{\mathbb{E}}[\int_{0}^{T}|\eta_{s}|^{p}ds]\biggr)^{1/p}$.
    Denote by $H_G^p(0,T)$, $M_{G}^{p}(0,T)$ the completion
	of $M_{G}^{0}(0,T)$ under the norms 
    $\small \|\cdot\|_{H_G^p}$, $\small \|\cdot\|_{M_G^p}$, respectively.
\end{definition}

The following proposition is the 
Burkholder-Davis-Gundy (BDG) inequality under $G$-expectation framework.
\begin{proposition}[\cite{P19}]\label{prop2.3}
	If $\eta\in H_G^{\alpha}(0,T)$ with $\alpha\geq 1$ and $p\in(0,\alpha]$, then we have
    \begin{small}
        \begin{equation*}
            \underline{\sigma}^p c_p\widehat{\mathbb{E}}\biggl[\biggl(\int_0^T |\eta_s|^2ds\biggr)^{p/2}\biggr]\leq
            \widehat{\mathbb{E}}\biggl[\sup_{t\in[0,T]}\biggl|\int_0^t\eta_s dB_s\biggr|^p\biggr]\leq
            \bar{\sigma}^p C_p\widehat{\mathbb{E}}\biggl[\biggl(\int_0^T |\eta_s|^2ds\biggr)^{p/2}\biggr],
        \end{equation*}
    \end{small}where $0<c_p<C_p<\infty$ are constants depending on $p, T$.
\end{proposition}

Now we state the monotone convergence theorem of $G$-framework, which is different from
the classical case.
\begin{theorem}[\cite{DHP11}]\label{monotone-con-thm}
    Let $X_n,n\geq1$ be a sequence of $\mathcal{B}(\Omega)$-measurable functions.
    \begin{enumerate}
        \item [(i)] Suppose $X_n\geq0$. Then 
        $\widehat{\mathbb{E}}[\liminf\nolimits_{n\to\infty} X_n]\leq \liminf\nolimits_{n\to\infty}\widehat{\mathbb{E}}[X_n].$
        \item [(ii)] Suppose $X_n\in L_G^1(\Omega_T)$ be such $X_n \downarrow X$ q.s.,
        then $\widehat{\mathbb{E}}[X_n]\downarrow \widehat{\mathbb{E}}[X]$.
        In particular, if $X\in L_G^1(\Omega_T)$, then $\widehat{\mathbb{E}}[Xn-X]\downarrow 0$, as $n\to\infty$.
    \end{enumerate}
\end{theorem}
\begin{lemma}[\cite{HTW22}]\label{lemma-HTW22}
    Let $X_n\in L_G^1(\Omega_T)$ for $n\geq1$ such that 
    $\sup\nolimits_{n\geq1}\widehat{\mathbb{E}}[|X_n|^{2p}]<\infty$ for some 
    $p\geq1$. If $\widehat{\mathbb{E}}[|X_n|]$ converges to $0$ as $n\to\infty,$
    then $\lim\nolimits_{n\to\infty}\widehat{\mathbb{E}}[|X_n|^{p}]=0.$
\end{lemma}

The following theorem is a version of the Doob's maximal inequality under $G$-framework.
More details could be found in \cite{P10,STZ12,Song11}.
%\begin{theorem}[\cite{LuSong21}]\label{Doob-inequality}
%    Let $1 \le \alpha < \alpha^\prime$. Then, for all $\xi\in L_G^{\alpha^\prime}(\Omega_T)$,
%    we can find a constant  $C=:C(\alpha,\alpha^\prime)>0$ such that
%    \begin{small}
%        \begin{equation*}
%            \mathbb{\widehat{E}}\Biggl[\sup_{t\in [0,T]}\widehat{\mathbb{E}}_{t}[|\xi|^{\alpha
%            }]\Biggr]\le C \biggl(\mathbb{\widehat{E}}\Bigl[|\xi|^{\alpha^\prime}\Bigr]\biggr)^{\frac{\alpha}{\alpha^\prime}}.
%        \end{equation*}
%    \end{small}
%    For any $\alpha \geq 1$ and $\delta>0$, we have 
%    $L_{G}^{\alpha+\delta}(\Omega_{T})\subset L_{\mathcal{E}}^{\alpha}(\Omega _{T}).$
%    More precisely,  for any $1 < \gamma<\beta :=(\alpha+\delta)/\alpha,\gamma\le 2$ 
%    and for all $\xi\in L_{ip}(\Omega_T)$, we have
%    \begin{equation*}
%        \mathbb{\widehat{E}}\Biggl[\sup_{t\in [0,T]}\widehat{\mathbb{E}}_{t}[|\xi|^{\alpha
%        }]\Biggr]\le C\Biggl\{\biggl(\widehat{\mathbb{E}}\Big[|\xi|^{\alpha+\delta}\Big]\biggr)^{\alpha/(\alpha+\delta
%        )}+\widehat{\mathbb{E}}\Big[|\xi|^{\alpha+\delta}\Big]\Biggr\}.
%    \end{equation*}
%\end{theorem}
\begin{theorem}\label{Doob-inequality}
    Let $1 \le \alpha < \beta$. Then for each $1<p<\bar{p}:=\beta/\alpha$ 
    with $p\leq 2$ and for all $\xi\in L_G^\beta(\Omega_T)$,
    we can find a constant $C$ depending only on $p,$ $\bar{\sigma}$ and $\underline{\sigma}$ such that
    \begin{small}
        \begin{equation*}
            \widehat{\mathbb{E}}\Biggl[\sup_{t\in [0,T]}\widehat{\mathbb{E}}_{t}[|\xi|^{\alpha
            }]\Biggr]\leq \frac{Cp\bar{p}}{(\bar{p}-p)(p-1)} 
            \biggl(\widehat{\mathbb{E}}\Bigl[|\xi|^\beta\Bigr]^{\frac{1}{p\bar{p}}}+\widehat{\mathbb{E}}\Bigl[|\xi|^\beta\Bigr]^{\frac{1}{p}}\biggr).
        \end{equation*}
    \end{small}
\end{theorem}
\begin{remark}[\cite{HTW22}]\label{Doob-exp}
    Suppose that $X$ is positive such that $e^X\in L_G^2(\Omega)$. Then, there exists a constant
    $\hat{A}_G$ depending only on $\bar{\sigma}$ and $\underline{\sigma}$ such that (taking $\alpha=2,\beta=4$)
    \begin{small}
        \begin{align*}
            \widehat{\mathbb{E}}\left[\sup_{t\in[0,T]}\widehat{\mathbb{E}}_t\left[e^X\right]\right]
            =\widehat{\mathbb{E}}\left[\sup_{t\in[0,T]}\widehat{\mathbb{E}}_t\left[\left(e^{\frac{X}{2}}\right)^2\right]\right]
            \leq \hat{A}_G\widehat{\mathbb{E}}\left[e^{2X}\right].
        \end{align*}
    \end{small}
\end{remark}

Let $S_G^0(0,T)=\{h(t,B_{t_1\wedge t}, \ldots,B_{t_n\wedge t}):t_1,\ldots,t_n\in[0,T],h\in C_{b,Lip}(\mathbb{R}^{n+1})\}$. 
For $p\geq 1$ and $\eta\in S_G^0(0,T)$, define 
$\small \|\eta\|_{S_G^p}=\tiny \biggl\{\widehat{\mathbb{E}}\biggl[\sup_{t\in[0,T]}|\eta_t|^p\biggr]\biggr\}^{1/p}$. 
Denote by $S_G^p(0,T)$ the completion of $S_G^0(0,T)$ under $\small \|\cdot\|_{S_G^p}$.  
Similarly, define $S_G^\infty(0,T)$ 
as the completion of $S_G^0(0,T)$ under $\|\eta\|_{S_G^\infty}:=\|\sup\nolimits{t\in[0,T]}|\eta_t|\|_{L_G^\infty}$.

We denote by $\mathcal{E}_G^p(\mathbb{R})$
the collection of all stochastic processes
$Y$ such that $e^{|Y|}\in S_G^p(0,T),$
$\mathcal{H}_G^p(\mathbb{R})$ the collection of all stochastic processes 
$Z\in H^p_G(0,T)$ and $\mathcal{L}_G^p(\mathbb{R})$
the collection of all stochastic processes $K$ such that $K$ is a
non-increasing G-martingale with $K_0 = 0$ and $K_T \in L^p_G(\Omega)$. 
We write $Y\in \mathcal{E}_G(\mathbb{R})$
if $Y\in \mathcal{E}_G^p(\mathbb{R})$
for any $p \geq 1$. Similarly, we define $\mathcal{H}_G(\mathbb{R})$ 
and $\mathcal{L}_G(\mathbb{R})$.

\subsection{$G$-BMO martingales and $G$-Girsanov Theorem}
We now review results on $G$-BMO martingale and $G$-Girsanov Theorem from
\cite{HLS18,PZ13}.
\begin{definition}\label{def2.2}
    A process $\{M_t\}$ with values in $L_G^1(\Omega_T)$ is called 
    a $G$-martingale (supermartingale,submartingale) if
    $\widehat{\mathbb{E}}_t[M_s]=M_s(\leq,\geq)$ for any $s\leq t$.
    Especially, The process $\{M_t\}$ is called a symmetric $G$-martingale
    if $-M$ is also a $G$-martingale.
\end{definition}

\begin{definition}\label{def2.3}
    For $Z\in H_G^2(0,T)$, a symmetric $G$-martingale $\int_0^\cdot Z_sdB_s$ on $[0,T]$ 
    is called a $G$-BMO martingale if
    \begin{small}
        \begin{equation*}
            \|Z\|_{\text{BMO}_G}^2:=\sup_{P\in\mathcal{P}}\|Z\|_{\text{BMO}(P)}^2
            =\sup_{P\in\mathcal{P}} \bigg[\sup_{\tau\in\mathcal{T}_0^T}
            \Biggl\|E_\tau^{P}\Big[\int_\tau^T|Z_t|^2d\langle B\rangle_t\Big]\Biggr\|_{L^\infty(P)} \bigg]<+\infty,
        \end{equation*}
    \end{small}where $\mathcal{T}_0^T$ denotes the totality of all 
    $\mathcal{F}$-stopping times taking values in $[0,T]$ 
    and $\|Z\|_{BMO(P)}$ stands for the BMO norm of 
    $\int_0^\cdot Z_sdB_s$ under probability measure $P$.
\end{definition}

The space $\text{BMO}_G$ is defined as 
\begin{small}
    \begin{align*}
        \text{BMO}_G:=\{Z\in H_G^2(0,T):\|Z\|_{\text{BMO}_G}<+\infty\}.
    \end{align*}
\end{small}We can define an exponential $G$-martingale via $G$-BMO martingale as in the classical case.
\begin{lemma}[\cite{HLS18}]\label{lemma2.1}
    For $Z\in \text{BMO}_G$, the process
    \begin{small}
        \begin{align*}
            \mathcal{E}(Z)_t:=\mathcal{E}\left(\int_0^t Z_s dB_s-\frac{1}{2}\int_0^t|Z_s|^2d\langle B\rangle_s \right), \quad t\ge 0
        \end{align*}
    \end{small}is a symmetric $G$-martingale.
\end{lemma}

The following general reverse H\"{o}lder's inequality plays an important role in
$G$-BMO martingale theory.
\begin{lemma}[\cite{CT20}]\label{reverse holder}
    (Reverse H\"{o}lder Inequality) Let 
    $\small \phi(x)=\Big\{1+\frac{1}{x^2}\log\frac{2x-1}{2(x-1)}\Big\}^{\frac{1}{2}}-1$ 
    and $1<q<+\infty.$ If $\|Z\|_{\text{BMO}_G}<\phi(q),$ we have
    \begin{small}
        \begin{equation*}
            \sup_{ P\in \mathcal{P}}\sup_{\tau\in\mathcal{T}_0^T}
            \Biggl\|E^{P}_\tau\Big[\Big\{\frac{\mathcal{E}(Z)_T}{\mathcal{E}(Z)_\tau}\Big\}^q\Big]\Biggr\|_{L^\infty( P)}\leq C_q
        \end{equation*}
    \end{small}
    for a constant $C_q>0$ depending only on $q$.
\end{lemma}

With the exponential $G$-martingale, we can generalize the $G$-Girsanov theorem. 
From \cite{HLS18}, we can define a new $G$-expectation 
$\tilde{\mathbb{E}}[\cdot]$ with $\mathcal{E}(Z)$ satisfying
\begin{small}
    \begin{equation*}%\label{tilde-E}
        \tilde{\mathbb{E}}[X]=\sup_{P\in\mathcal{P}}E^P[\mathcal{E}(Z)_TX]
        =\widehat{\mathbb{E}}[\mathcal{E}(Z)_TX],\quad \forall X\in L^p_G(\Omega_T),
    \end{equation*}    
\end{small}where $\small p>\frac{q}{q-1}$ and $q$ is the order in the reverse 
H\"{o}lder inequality for $\mathcal{E}(Z)$.
Moreover, the conditional expectation $\tilde{\mathbb{E}}_t[\cdot]$ 
is defined as
\begin{small}
    \begin{equation*}
        \tilde{\mathbb{E}}_t[X]=\widehat{\mathbb{E}}_t\Big[\frac{\mathcal{E}(Z)_T}{\mathcal{E}(Z)_t}X\Big],~~\text{q.s.},~~\forall X\in L^p_G(\Omega_T).
    \end{equation*}
\end{small}

The following lemmas give the Girsanov theorem in the $G$-framework.
\begin{lemma}[\cite{HLS18}]\label{lemma-Gisr-1}
    Suppose that $Z\in \text{BMO}_G$. We define a new $G$-expectation 
    $\tilde{\mathbb{E}}[\cdot]$ by $\mathcal{E}(Z)$. Then the process 
    $B-\int Zd\langle B\rangle$ is a $G$-Brownian motion under $\tilde{\mathbb{E}}[\cdot]$.
\end{lemma}
    
\begin{lemma}[\cite{HLS18}]\label{lemma-Gisr-2}
    Suppose that $Z\in \text{BMO}_G$. We define a new $G$-expectation 
    $\tilde{\mathbb{E}}[\cdot]$ by $\mathcal{E}(Z)$. Suppose that $K$ is a decreasing 
    $G$-martingale such that $K_0=0$ and for some 
    $p>\frac{q}{q-1}, K_t\in L_G^p(\Omega_t), 0\leq t\leq T,$ 
    where $q$ is the order in the reverse H\"{o}lder inequality 
    for $\mathcal{E}(Z)$. Then $K$ is a decreasing 
    $G$-martingale under $\tilde{\mathbb{E}}[\cdot]$.
\end{lemma}

\subsection{Backward Skorokhod problem with two nonlinear reflecting boundaries}
In this subsection, we recall the backward Skorokhod problem 
studied in \cite{Li24}, which is important for the construction 
of solutions to $G$-BSDEs with double mean reflections. 
We introduce the following notations.

\begin{enumerate}
    \item[$\bullet$] $C[0,T]$: the set of continuous functions from $[0,T]$ to $\mathbb{R}$.
    \item[$\bullet$] $BV[0,T]$: the set of functions in $C[0,T]$ starting from the origin with bounded variation on $[0,T]$.
    \item[$\bullet$] $I[0,T]$: the set of functions in $C[0,T]$ starting from the origin which is nondecreasing.
\end{enumerate}

To proceed, we first recall the Skorokhod problem studied in \cite{Li23}.
\begin{definition}\label{def-solution-SP}
Let $s\in C[0,T]$,  and $l,r:[0,T] \times \mathbb{R}\rightarrow \mathbb{R}$ be 
two functions with $l\leq r$.  A pair of functions $(x,K)\in C[0,T]\times BV[0,T]$ 
is called a solution to the  Skorokhod problem for $s$ with 
nonlinear constraints $l,r$ ($(x,K)=\mathbb{SP}_l^r(s)$ for short) 
if 
\begin{enumerate}
    \item[(i)] $x_t=s_t+K_t$;
    \item[(ii)] $l(t,x_t)\leq 0\leq r(t,x_t)$;
    \item[(iii)]  $K_{0-}=0$ and $K$ has the decomposition $K=K^r-K^l$, where $K^r,K^l$ are nondecreasing functions satisfying  
    \begin{small}
        \begin{align*}
            \int_0^{T} I_{\{l(s,x_s)<0\}}dK^l_s=0, \  \int_0^{T} I_{\{r(s,x_s)>0\}}dK^r_s=0.
        \end{align*}
    \end{small}
\end{enumerate}
\end{definition}

We propose the following assumption on the functions $l,r$.
\begin{assumption}\label{assumption-l-r}
    The functions $l,r:[0,T]\times \mathbb{R}\rightarrow \mathbb{R}$ 
    satisfy the following conditions:
    \begin{enumerate}
    \item[(i)] For each fixed $x\in\mathbb{R}$, $l(\cdot,x),r(\cdot,x)\in C[0,T]$;
    \item[(ii)] For any fixed $t\in[0,T]$, $l(t,\cdot)$, $r(t,\cdot)$ are strictly increasing;
    \item[(iii)] There exists two positive constants $0<c<C<\infty$, 
    such that for any $t\in [0,T]$ and $x,y\in \mathbb{R}$,
    \begin{small}
        \begin{align*}
            &c|x-y|\leq |l(t,x)-l(t,y)|\leq C|x-y|,\\
            &c|x-y|\leq |r(t,x)-r(t,y)|\leq C|x-y|.
            \end{align*}
    \end{small}
    \item[(iv)] $\inf_{(t,x)\in[0,T]\times\mathbb{R}}(r(t,x)-l(t,x))>0$.
    \end{enumerate}
\end{assumption}

\begin{theorem}[\cite{Li24}]\label{thm-SP}
Suppose that $l,r$ satisfy Assumption \ref{assumption-l-r}. 
For any given $s\in C[0,T]$, there exists a unique pair of solution 
to the Skorokhod problem $(x,K)=\mathbb{SP}_l^r(s)$.
\end{theorem}

The following proposition provides the continuity property of 
the solution to the Skorokhod problem with respect to input 
function $s$ and the reflecting boundary functions $l,r$.

\begin{theorem}[\cite{Li24}]\label{thm-SP-K1-K2}
Suppose that $(l^i,r^i)$ satisfy Assumption \ref{assumption-l-r}, $i=1,2$.
Given  $s^i\in C[0,T]$, let $(x^i,K^i)$ be the solution to the Skorokhod problem $\mathbb{SP}_{l^i}^{r^i}(s^i)$. Then, we have
\begin{small}
    \begin{equation*}%\label{diff-SP-k}
        \sup_{t\in[0,T]}\left|K^1_t-K^2_t\right|
        \leq \frac{C}{c}\sup_{t\in[0,T]}\left|s^1_t-s^2_t\right|+\frac{1}{c}(\bar{L}_T\vee\bar{R}_T),
        \end{equation*}
\end{small}where
\begin{small}
    \begin{align*}
        \bar{L}_T=\sup_{(t,x)\in[0,T]\times \mathbb{R}}\left|{l}^1(t,x)-{l}^2(t,x)\right|,\
        \bar{R}_T=\sup_{(t,x)\in[0,T]\times \mathbb{R}}\left|{r}^1(t,x)-{r}^2(t,x)\right|.
        \end{align*}
\end{small}
\end{theorem}

We give the definition of a solution to the backward Skorokhod problem.
\begin{definition}\label{def-solution-BSP}
    Let $s\in C[0,T]$, $a\in \mathbb{R}$ and 
    $l,r:[0,T]\times \mathbb{R}\rightarrow \mathbb{R}$ 
    be two functions such that $l\leq r$ and $l(T,a)\leq 0\leq r(T,a)$. 
    A pair of functions $(x,k)\in C[0,T]\times BV[0,T]$ is called a 
    solution of the backward Skorokhod problem for $s$ 
    with nonlinear constraints $l,r$ ($(x,k)=\mathbb{BSP}_l^r(s,a)$ for short) 
    if
    \begin{enumerate}
    \item[(i)] $x_t=a+s_T-s_t+k_T-k_t$;
    \item[(ii)] $l(t,x_t)\leq 0\leq r(t,x_t)$, $t\in[0,T]$;
    \item[(iii)] $k_{0}=0$ and $k$ has the decomposition $k=k^r-k^l$, 
    where $k^r,k^l\in I[0,T]$ satisfy 
    \begin{small}
        \begin{align*}%\label{kl-kr}
            \int_0^T l(s,x_s)dk^l_s=0, \ \int_0^T r(s,x_s)dk^r_s=0.
        \end{align*}
    \end{small}
    \end{enumerate}
\end{definition}

\begin{theorem}[\cite{Li24}]\label{thm-BSP}
    Let Assumption \ref{assumption-l-r} hold. For any given $s\in C[0,T]$ 
    and $a\in \mathbb{R}$ with $l(T,a)\leq 0\leq r(T,a)$, 
    there exists a unique solution to the backward Skorokhod 
    problem $(x,k)=\mathbb{BSP}_l^r(s,a)$.
\end{theorem}

The following theorem reveals the continuous dependence 
of the solution to the backward Skorokhod problem with 
respect to the terminal value, the input function and the reflecting boundary functions.
    
\begin{theorem}[\cite{Li24}]\label{thm-BSP-K1-K2}
    Given $a^i\in\mathbb{R}$, $s^i\in C[0,T]$, $l^i,r^i$ satisfying 
    Assumption \ref{assumption-l-r} and $l^i(T,a^i)\leq 0\leq r^i(T,a^i)$, $i=1,2$,  
    let $(x^i,k^i)$ be the solution to the backward 
    Skorokhod problem $\mathbb{BSP}_{l^i}^{r^i}(s^i,a^i)$. Then, we have
    \begin{small}
        \begin{equation*}%\label{diff-BSP-K}
            \sup_{t\in[0,T]}|k^1_t-k^2_t|
            \leq  2\frac{C}{c}|a^1-a^2|+4\frac{C}{c}\sup_{t\in[0,T]}
            |s^1_t-s^2_t|+\frac{2}{c}(\bar{L}_T\vee\bar{R}_T),
        \end{equation*}
    \end{small}where
    \begin{small}
        \begin{align*}
            \bar{L}_T=\sup_{(t,x)\in[0,T]\times \mathbb{R}}|{l}^1(t,x)-{l}^2(t,x)|,\ 
            \bar{R}_T=\sup_{(t,x)\in[0,T]\times \mathbb{R}}|{r}^1(t,x)-{r}^2(t,x)|.
        \end{align*}
    \end{small}
\end{theorem}

\section{Quadratic $G$-BSDE with double constraints}\label{sec3}
In this section, we investigate $G$-BSDEs with double mean reflections,
where the generator has quadratic growth in $z$.

Specifically, $G$-BSDE with double mean reflections is given by 
\begin{small}
    \begin{equation}\label{DMFGBSDE}
        \begin{cases}
            Y_t=\xi+\int_t^T f(s,Y_s,Z_s)ds+\int_t^T g(s,Y_s,Z_s)d\langle B\rangle_s-\int_t^T Z_s dB_s-(K_T-K_t)+(A_T-A_t), \\
            \widehat{\mathbb{E}}[L(t,Y_t)]\leq 0\leq \widehat{\mathbb{E}}[R(t,Y_t)], \\
            A_t=A^R_t-A^L_t \textrm{ and } \int_0^T \widehat{\mathbb{E}}[R(t,Y_t)]dA_t^R=\int_0^T \widehat{\mathbb{E}}[L(t,Y_t)]dA^L_t=0.
        \end{cases}
    \end{equation}
\end{small}We impose the following assumptions on $L,R$:
\begin{assumption}\label{assumption-L-R}
    The running loss functions $L,R:\Omega_T\times [0,T]\times\mathbb{R}\rightarrow \mathbb{R}$ satisfy the following conditions:
    \begin{enumerate}
    \item[(i)] $(t,x)\rightarrow L(\omega,t,x)$, $(t,x)\rightarrow R(\omega,t,x)$ are uniformly continuous, uniform in $\omega$;
    \item[(ii)]  for any fixed $(t,x)\in [0,T]\times \mathbb{R}$, $L(t,x),R(t,x)\in L_G^1(\Omega_T)$;
    \item[(iii)] for any fixed $(\omega,t)\in \Omega_T\times [0,T]$, $L(\omega,t,\cdot),R(\omega,t,\cdot)$ are strictly increasing and there exists two constants $0<c<C<\infty$ such that for any $x,y\in \mathbb{R}$,
    \begin{small}
        \begin{align*}
            &c|x-y|\leq |L(\omega,t,x)-L(\omega,t,y)|\leq C|x-y|,\\
            &c|x-y|\leq |R(\omega,t,x)-R(\omega,t,y)|\leq C|x-y|;
            \end{align*}
    \end{small}
    \item[(iv)] for any fixed $(\omega,t)\in \Omega_T\times [0,T]$, there exists a constant $M\geq 0$ such that
    \begin{small}
        \begin{align*}
            |L(t,y)|\leq M(1+|y|), \ |R(t,y)|\leq M(1+|y|);
            \end{align*}
    \end{small}
    \item[(v)] $\inf_{(\omega,t,x)\in \Omega_T\times [0,T]\times\mathbb{R}} (R(\omega,t,x)-L(\omega,t,x))>0$.
    \end{enumerate}
\end{assumption}
We will analyze the cases of bounded and unbounded terminal conditions separately. 
For the first case, we use the $G-\text{BMO}$ martingale theory and a 
fixed-point argument. For the second case, we need to apply 
the $\theta$-method and introduce either convexity or concavity 
on the generator.

\subsection{Bounded terminal condition}
The goal of this part is to establish the existence and uniqueness of solution to 
the doubly mean reflected $G$-BSDE \eqref{DMFGBSDE} with quadratic 
generators and bounded terminal condition.  
In order to do this, we make the following assumptions.
\begin{assumption}\label{assumption-BT}
    Assume that the following hold:
    \begin{enumerate}
        \item[(i)] The terminal condition $\xi\in L_G^{\infty}(\Omega_T)$ with $\widehat{\mathbb{E}}[L(T,\xi)]\leq 0\leq \widehat{\mathbb{E}}[R(T,\xi)]$;
        \item[(ii)] $\small \int_0^T|f(t,\omega,0,0)|^2dt+\int_0^T|g(t,\omega,0,0)|^2dt +|\xi(\omega)|\leq M_0$, q.s.;
        \item[(iii)] The generator $(f,g)$ is uniformly continuous in $(t,\omega)$, i.e. there is a non-decreasing continuous function $w: [0, +\infty)\to [0, +\infty)$ such that $w(0)=0$ and
        \begin{small}
            \begin{align*}
                &\sup_{y,z\in \mathbb{R}}|f(t,\omega,y,z)-f(t^\prime,\omega^\prime,y,z)|\leq w(|t-t^\prime|+\|\omega-\omega^\prime\|_\infty),\\
                &\sup_{y,z\in \mathbb{R}}|g(t,\omega,y,z)-g(t^\prime,\omega^\prime,y,z)|\leq w(|t-t^\prime|+\|\omega-\omega^\prime\|_\infty);
            \end{align*}
        \end{small}
        \item[(iv)] There are two positive constants $L_1$ and $L_2$ such that for each $(t,\omega)\in [0,T]\times\Omega,$
        \begin{small}
            \begin{align*}
                &|f(t,\omega,y,z)-f(t,\omega,y^\prime,z^\prime)|+|g(t,\omega,y,z)-g(t,\omega,y^\prime,z^\prime)|\\
                &\leq L_1|y-y^\prime|+L_2(1+|z|+|z^\prime|)|z-z^\prime|.
            \end{align*}
        \end{small}
    \end{enumerate}
\end{assumption}

Without loss of generality, we assume $f\equiv0$ in \eqref{DMFGBSDE} for simplicity.
For given $U\in S_G^\infty$, consider the following type $G$-BSDE:
\begin{small}
    \begin{align}\label{MFGBSDE-Y^U}
        \begin{cases}
            Y^U_t=\xi+\int_t^T g\left(s,U_{s},Z^U_{s}\right)d\langle B\rangle_s-\int_t^T Z^U_s dB_s-\left(K^U_T-K^U_t\right)+\left(A^U_T-A^U_t\right),\ t\in[0,T],\\
            \widehat{\mathbb{E}}\left[L(t,Y^U_t)\right]\leq 0 \leq \widehat{\mathbb{E}}\left[R(t,Y^U_t)\right], \ A^U_t=A_t^{U,R}-A_t^{U,L},\ A_t^{U,R},A_t^{U,L}\in I[0,T],  \\
            \int_{0}^{T}\widehat{\mathbb{E}}\left[R(t,Y^U_t)\right]dA_t^{U,R}=\int_{0}^{T}\widehat{\mathbb{E}}\left[L(t,Y^U_t)\right]dA_t^{U,L}=0.
        \end{cases}
    \end{align}
\end{small}
\begin{lemma}\label{lemma-MFGBSDE-Y^U}
    Suppose that Assumptions \ref{assumption-L-R} and \ref{assumption-BT} hold.
    Then, the quadratic $G$-BSDE \eqref{MFGBSDE-Y^U} with double mean reflections 
    admits a unique solution 
    $(Y,Z,K,A)\in S_G^{\infty}\times\text{BMO}_G\times \cap_{\alpha\ge2}S_G^\alpha\times \mathcal{A}_D$, 
    where $\mathcal{A}_D$ is the collection of deterministic, non-decreasing 
    and continuous processes $(R_t)_{0\leq t\leq T}$ with $R_0 =0$.
\end{lemma}

\begin{proof}
    We first show the uniqueness. Let $(Y^i,Z^i,K^i,A^i),i=1,2$ be two solutions 
    of \eqref{MFGBSDE-Y^U}. By Theorem 5.1 in \cite{CT20} (with $S=-\infty$),  
    $(Y^i-(A^i_T-A^i),Z^i,K^i)$ 
    can be seen as the solution to the $G$-BSDE with data $(\xi,\{g(s,U_s,z)\}_{s\in[0,T]}),$
    then we have $Y^1-(A^1_T-A^1)\equiv Y^2-(A^2_T-A^2)$, 
    $Z^1\equiv Z^2(=:z)$ and $K^1\equiv K^2$. The  uniqueness of $A^1\equiv A^2$
    follows from the same argument as Proposition 3.4 in \cite{HL24} (setting $C_s=0$ and $D_s=g(s,U_s,Z_s)$),
    so we omit it. 

    Turn to the existence, by Theorem 5.1 in \cite{CT20} (with $S=-\infty$), the following $G$-BSDE
    \begin{small}
        \begin{align*}
            y^U_t=\xi+\int_t^T g(s,U_s,z^U_s)d\langle B\rangle_s-\int_t^T z^U_s dB_s-(k_T^U-k_t^U)
        \end{align*}
    \end{small}has a unique solution $(y^U,z^U,k^U)\in S_G^\infty\times \text{BMO}_G\times \cap_{\alpha\ge2}S_G^\alpha$.
    Define 
    \begin{small}
        \begin{align}\label{slr}
            &s^U_t=\widehat{\mathbb{E}}\left[\xi+\int_0^T g(s,U_s,z^U_s) d\langle B\rangle_s\right]-
            \widehat{\mathbb{E}}\left[\xi+\int_t^T g(s,U_s,z^U_s) d\langle B\rangle_s\right]=\widehat{\mathbb{E}}\left[y^U_0\right]-\widehat{\mathbb{E}}\left[y^U_t\right], \ a=\widehat{\mathbb{E}}\left[\xi\right],\notag\\
            &l^U(t,x):=\widehat{\mathbb{E}}\left[L(t,y^U_t-\widehat{\mathbb{E}}\left[y^U_t\right]+x)\right], \ r^U(t,x):=\widehat{\mathbb{E}}\left[R(t,y^U_t-\widehat{\mathbb{E}}\left[y^U_t\right]+x)\right].
        \end{align}
    \end{small}By Lemma 3.3 in \cite{HL24}, we can check that $l^U,r^U$ satisfy Assumption \ref{assumption-l-r} and
    the following hold:
    \begin{small}
        \begin{align*}
            l^U(T,a)=\widehat{\mathbb{E}}\left[L(T,\xi)\right]\leq 0\leq \widehat{\mathbb{E}}\left[R(T,\xi)\right]=r^U(T,a).
        \end{align*}
    \end{small}By Theorem \ref{thm-BSP}, there exists a unique $(x^U,A^U)$ to solve backward Skorokhod problem 
    $\mathbb{BSP}_{l^U}^{r^U}(s^U,a)$ with $A^U_t=\bar{A}^U_T-\bar{A}^U_{T-t}$, 
    where $\bar{A}^U$ is the second component of the solution to 
    the Skorokhod problem $\mathbb{SP}_{\bar{l}^U}^{\bar{r}^U}(\bar{s}^U)$, and 
    \begin{small}
        \begin{equation}\label{bar-slr}
            \bar{s}^U_{t}=a+s^U_{T}-s^U_{T-t}=\widehat{\mathbb{E}}[y^U_{T-t}],\ \bar{l}^U(t,x)=l^U (T-t,x), \ \bar{r}^U(t,x)=r^U (T-t,x).
        \end{equation}
    \end{small}Then, define 
    \begin{small}
        \begin{align}\label{Y-Z-K-A-U}
            Y^U_t=y^U_t+A^U_T-A^U_t=y^U_t+\bar{A}^U_{T-t}, \ Z^U_t=z^U_t, \ K_t^U=k_t^U, \ t\in[0,T].
        \end{align}
    \end{small}Finally, we can verify that the Skorokhod conditions are satisfied by similar analysis as 
    Proposition 3.4 in \cite{HL24}.
    Hence, $(Y^U,Z^U,K^U,A^U)$ is a unique solution of equation \eqref{MFGBSDE-Y^U}.
\end{proof}
\begin{remark}
    Relation \eqref{Y-Z-K-A-U} suggests that 
    the solution of double mean-reflected $G$-BSDE can be represented as the sum of a solution to standard $G$-BSDE and 
    a solution to the Skorokhod problem. This plays a key role in the subsequent proof of uniqueness and also helps us 
    to estimate double mean-reflected $G$-BSDE by the existing estimates of standard $G$-BSDE and the Skorokhod
    problem.  
\end{remark}

\begin{theorem}\label{thm-BT}
    Suppose that Assumptions \ref{assumption-L-R} and \ref{assumption-BT} hold.
    Then, the quadratic $G$-BSDE \eqref{DMFGBSDE} with double mean reflections 
    admits a unique solution $(Y,Z,K,A)\in S_G^{\infty}\times\text{BMO}_G\times \cap_{\alpha\ge2}S_G^\alpha\times \mathcal{A}_D$.
\end{theorem}

\begin{proof} 
    We first consider the case that $T\leq \delta$, where $\delta$ is a constant
    small enough to be determined later.
    $M$ will denote a positive constant only depending on 
    some parameters and may change from line to line.
    The proof will be divided into the following two steps.

    {\bf Step 1: Contraction Mapping on Small Intervals.} For $U^i\in S_G^\infty,i=1,2$, we define a mapping 
    $\Gamma:S_G^\infty\rightarrow S_G^\infty$ as follows
    \begin{small}
        \begin{equation*}
            \Gamma(U^i):=Y^{U,i},\ i=1,2,
            \end{equation*}
    \end{small}where $Y^{U,i}$ is the first component of the solution to \eqref{MFGBSDE-Y^U}.

    For each $U^i\in S_G^\infty$, by \eqref{Y-Z-K-A-U}, we have
    \begin{small}
        \begin{align*}
            (Y^{U,i}_t,Z^{U,i}_t,K^{U,i}_t)=(y_t^i+(A^{U,i}_T-A^{U,i}_t),z_t^i,k_t^i)
        \end{align*}
    \end{small}where $(y_t^i,z_t^i,k_t^i)$ is the solution to the following $G$-BSDE:
    \begin{small}
        \begin{align*}
            y_t^i=\xi+\int_t^T g(s,U_s^i, z^i_s)d\langle B\rangle_s-\int_t^T z^i_s dB_s-(k_T^i-k_t^i).
        \end{align*} 
    \end{small}Denote $\hat{\psi}_t=\psi_t^1-\psi_t^2,\psi=y,z,k,U$. For each $t \in[0, T]$,
    define
    \begin{small}
        \begin{equation*}
            \hat{b}_t=\frac{g\left(t, U_t^1,z_t^1\right)-g\left(t, U_t^1,z_t^2\right)}{|\hat{z}_t|^2}\hat{z}_t \mathbf{1}_{\left\{|\hat{z}_t| \neq 0\right\}} , 
        \end{equation*}
    \end{small}then
    \begin{small}
        \begin{align}\label{hat-y-1}
            \hat{y}_t=  \int_t^T\left(\hat{b}_s \hat{z}_s+g\left(s,U_s^1, z_s^2\right)
            -g\left(s,U_s^2, z_s^2\right)\right) d\langle B\rangle_s 
             -\int_t^T \hat{z}_s d B_s-\int_t^Tdk_s^1+\int_t^Tdk_s^2.
        \end{align}
    \end{small}Using Assumption \ref{assumption-BT}, it is easy to check that $\hat{b}\in\text{BMO}_G$ and 
    $|\hat{b}_s|\le L_2(1+|z_s^1|+|z_s^2|)$.
    By Lemma \ref{lemma2.1}, $\mathcal{E}(\hat{b})_t$ is a symmetric $G$-martingale defined as the following: 
    \begin{small}
        \begin{align*}
            \mathcal{E}(\hat{b})_t:=\exp\left(\int_{0}^{t}\hat{b}_sd B_s-
            \frac{1}{2}\int_{0}^{t}|\hat{b}_s|^2d\langle B\rangle_s\right),\ t\ge0,
        \end{align*}
    \end{small}From \cite{HLS18}, we can define a new $G$-expectation with $\mathcal{E}(\hat{b})$ 
    satisfying 
    \begin{small}
        \begin{equation*}
            \tilde{\mathbb{E}}[X]=\widehat{\mathbb{E}}[\mathcal{E}(\hat{b})_TX],\quad \forall X\in L^p_G(\Omega_T),
        \end{equation*}
    \end{small}where $\small p>\frac{q}{q-1}$ and $q$ is the order in the reverse H\"{o}lder inequality
    for $\mathcal{E}(\hat{b})$. Then, the conditional expectation $\tilde{\mathbb{E}}_t[\cdot]$ 
    is
    \begin{small}
        \begin{equation*}
            \tilde{\mathbb{E}}_t[X]=\widehat{\mathbb{E}}_t\Big[\frac{\mathcal{E}(\hat{b})_T}{\mathcal{E}(\hat{b})_t}X\Big],~~\text{q.s.},~~\forall X\in L^p_G(\Omega_T).
        \end{equation*}
    \end{small}By Lemma \ref{lemma-Gisr-1}, we know that $\tilde{B}:=B-\int \hat{b}d\langle B\rangle$
    is a $G$-Brownian motion under $\tilde{\mathbb{E}}$.
    Thus, \eqref{hat-y-1} rewrites as:
    \begin{small}
        \begin{align}\label{hat-y-2}
            \hat{y}_t=  \int_t^T\left(g\left(s,U_s^1, z_s^2\right)
            -g\left(s,U_s^2, z_s^2\right)\right) d\langle B\rangle_s 
             -\int_t^T \hat{z}_s d \tilde{B}_s-\int_t^Tdk_s^1+\int_t^Tdk_s^2 .
        \end{align}
    \end{small}Moreover, by Lemma \ref{lemma-Gisr-2}, $k^1$ and $k^2$ are two non-increasing $G$-martingales
    under $\tilde{\mathbb{E}}$. 
    And by $y_t\in S_G^\infty$, we have $y_t\in L^p_G(\Omega_T)$ for each $p$.
    Taking conditional expectation $\tilde{\mathbb{E}}_t$ on both sides of \eqref{hat-y-2}, we get 
    \begin{small}
        \begin{align}\label{estimate-hat-y}
            |\hat{y}_t|&\leq \tilde{\mathbb{E}}_t\left[
                \int_{t}^{T}|g\left(s,U_s^1, z_s^2\right)
                -g\left(s,U_s^2, z_s^2\right)| d\langle B\rangle_s\right]\notag\\
                &\leq \tilde{\mathbb{E}}_t\left[
                    \int_{t}^{T}L_1|\hat{U}| d\langle B\rangle_s
                \right]\\
                &\leq \widehat{\mathbb{E}}_t\left[
                    \frac{\mathcal{E}(\hat{b})_T}{\mathcal{E}(\hat{b})_t}\int_{t}^{T}L_1|\hat{U}| d\langle B\rangle_s
                \right]\notag\\
                &\leq L_1T\bar{\sigma}\|\hat{U}\|_{S_G^\infty}\notag.
        \end{align}
    \end{small}For simplicity, let 
    $s^i=s^{U^i}$, $l^i=l^{U^i}$, $r^i=r^{U^i}$ 
    and define $\bar{s}^i$, $\bar{l}^i$, $\bar{r}^i$ ,$i=1,2$ as in \eqref{bar-slr}. 
    Let $\bar{A}^i$ be the solution to the Skorokhod 
    problem $\mathbb{SP}_{\bar{l}^i}^{\bar{r}^i}(\bar{s}^i)$, $i=1,2$. 
    By Theorem \ref{thm-SP-K1-K2}, we obtain 
    \begin{small}
        \begin{align*}
            \sup_{t \in[0, T]}\left|\bar{A}^1_{t}-\bar{A}^2_{t}\right|\leq 
            \frac{C}{c}\sup_{t\in[0,T]}\left|\bar{s}^1_t-\bar{s}^2_t\right|+\frac{1}{c}(\bar{L}_T \vee \bar{R}_T),	
        \end{align*}
    \end{small}where
    $\bar{L}_T=\sup_{(t,x)\in[0,T]\times \mathbb{R}}|\bar{l}^1(t,x)-\bar{l}^2(t,x)|$, $
    \bar{R}_T=\sup_{(t,x)\in[0,T]\times \mathbb{R}}|\bar{r}^1(t,x)-\bar{r}^2(t,x)|$.
    Note that $\bar{s}^i_{t}=\widehat{\mathbb{E}}[y^i_{T-t}],$
    we get 
    $\sup\nolimits_{t\in[0,T]}\left|\bar{s}^1_t-\bar{s}^2_t\right|\leq \sup\nolimits_{t\in[0,T]} \widehat{\mathbb{E}}\left[|\hat{y}_t|\right].$
    And by Assumption \ref{assumption-L-R}, we have $\bar{L}_T\leq 2C\sup\nolimits_{t \in[0, T]}\widehat{\mathbb{E}}[|\hat{y}_t|]$ 
    and $\bar{R}_T\leq 2C\sup\nolimits_{t \in[0, T]}\widehat{\mathbb{E}}[|\hat{y}_t|]$.
    \noindent Hence, we obtain that
    \begin{small}
        \begin{equation}\label{hat-bar-A}
            \sup_{t \in[0, T]}\left|\bar{A}^1_{t}-\bar{A}^2_{t}\right|
            \leq \frac{3C}{c}\sup_{t \in[0, T]}\widehat{\mathbb{E}}\left[|\hat{y}_t|\right].
        \end{equation}
    \end{small}Then, Combining \eqref{Y-Z-K-A-U}, \eqref{estimate-hat-y} and \eqref{hat-bar-A}, 
    we can derive that
    \begin{small}
        \begin{align*}
            \left\|\Gamma\left(U^1\right)-\Gamma\left(U^2\right)\right\|_{S_G^{\infty}}
            &=\left\| \hat{y}+\bar{A}^1-\bar{A}^2\right\|_{S_G^{\infty}}
            \leq \left\| \hat{y}\right\|_{S_G^{\infty}}+\sup_{t\in[0,T]}
            \left|\bar{A}^1_t-\bar{A}^2_t\right| \\
            &\leq M(c,C,L_1,\bar{\sigma})T\|\hat{U}\|_{S_G^{\infty}},
        \end{align*}
    \end{small}where $M(c,C,L_1,\bar{\sigma})=\left(1+\frac{3C}{c}\right)L_1\bar{\sigma}$.

    Therefore, if $\delta$ is sufficiently small such that 
    $M(c,C,L_1,\bar{\sigma})\delta<1$, $\Gamma$ defines a strict contraction map on $[0,T]$,
    which implies the existence and the uniqueness of the 
    solution on sufficiently small time interval $[0,T]$.
     
    {\bf Step 2: Global Solution via Backward Induction.} For the general $T$, we choose $n\geq 1$ such 
    that $n\delta \geq T$. Set $T_k:=\frac{k T}{n}$, for $k=0,1, \cdots, n$. 
    By backward induction, for $k=n, n-1, \cdots, 1$, 
    there exists a unique solution $\left(Y^k, Z^k, K^k,A^k\right)$ 
    to the following $G$-BSDE with double mean reflections on 
    the interval $\left[T_{k-1}, T_k\right]$
    \begin{small}
        \begin{equation*}
            \begin{cases}
                Y_t^k=Y_{T_k}^{k+1}+\int_t^{T^k} g(s,Y_s^k,Z_s^k)d
                \langle B\rangle_s-\int_t^{T^k} Z_s^k dB_s-(K_{T_k}^k-K_t^k)+(A_{T_k}^k-A_t^k), \\
                \widehat{\mathbb{E}}[L(t,Y_t^k)]\leq 0\leq \widehat{\mathbb{E}}[R(t,Y_t^k)],\ t\in \left[T_{k-1}, T_k\right], \\
                A_{T_{k-1}}^k=0, A_t^k=A^{k,R}_t-A^{k,L}_t \textrm{ and } 
                \int_{T_{k-1}}^{T_k} \widehat{\mathbb{E}}[R(t,Y_t^k)]dA_t^{k,R}
                =\int_{T_{k-1}}^{T_k} \widehat{\mathbb{E}}[R(t,Y_t^k)]dA_t^{k,L}=0,
            \end{cases}
        \end{equation*}
    \end{small}where $Y_T^{n+1}=Y_T^{n}=\xi$. We denote
    \begin{small}
        \begin{align*}
            &Y_t=\sum_{k=1}^n Y_t^k I_{[T_{k-1}, T_k)}(t)+Y_T^{n} I_{\{T\}}(t), 
            \  Z_t=\sum_{k=1}^n Z_t^k I_{[T_{k-1}, T_k)}(t)+Z_T^{n} I_{\{T\}}(t),\\
            &K_t=K_t^k+\sum_{j=1}^{k-1} K_{T_j}^j, 
            \ A_t=A_t^k+\sum_{j=1}^{k-1} A_{T_j}^j, 
            \ t\in\left[T_{k-1}, T_k\right], \ k=1,\cdots, n,
        \end{align*}
    \end{small}with the notation $\sum_{j=1}^{0} K_{T_j}^j=\sum_{j=1}^{0} A_{T_j}^j=0$ 
    ($A^R, A^L$ are defined similarly). 
    It is easy to check that 
    $(Y,Z,K,A)\in S_G^{\infty}\times\text{BMO}_G\times \cap_{\alpha\ge2}S_G^\alpha\times \mathcal{A}_D$ 
    is a solution to quadratic $G$-BSDE \eqref{DMFGBSDE} with double mean reflections. 
    Uniqueness on the whole interval is a direct consequence of the uniqueness on each
    small interval.

    The proof is complete.
\end{proof}

\subsection{Unbounded terminal condition}
In this subsection, we aim to investigate the existence and uniqueness of solution to 
the doubly mean reflected G-BSDE \eqref{DMFGBSDE} with quadratic generators 
and unbounded terminal condition. 
Without loss of generality, we simplify \eqref{DMFGBSDE} as follows:
\begin{small}
    \begin{equation}\label{DMFGBSDE-simplify}
        \begin{cases}
            Y_t=\xi+\int_t^T g(s,Y_s,Z_s)d\langle B\rangle_s-\int_t^T Z_s dB_s-(K_T-K_t)+(A_T-A_t), \\
            \widehat{\mathbb{E}}[L(t,Y_t)]\leq 0\leq \widehat{\mathbb{E}}[R(t,Y_t)], \\
            A_t=A^R_t-A^L_t \textrm{ and } \int_0^T \widehat{\mathbb{E}}[R(t,Y_t)]dA_t^R=\int_0^T \widehat{\mathbb{E}}[L(t,Y_t)]dA^L_t=0.
        \end{cases}
    \end{equation}
\end{small}

For this, we make the following assumptions.
\begin{assumption}\label{assumption-UBT}
    Assume that the following hold:
    \begin{enumerate}
        \item[(i)] The terminal condition $\xi\in L_G^1(\Omega)$ with $\widehat{\mathbb{E}}[L(T,\xi)]\leq 0\leq \widehat{\mathbb{E}}[R(T,\xi)]$;
        \item[(ii)] There are two positive constants $L_1$ and $L_2$ such that for each $(t,\omega)\in [0,T]\times\Omega,$
        \begin{small}
            \begin{align*}
                |g(t,\omega,y,z)-g(t,\omega,y^\prime,z^\prime)|
                \leq L_1|y-y^\prime|+L_2(1+|z|+|z^\prime|)|z-z^\prime|;
            \end{align*}
        \end{small}
        \item[(iii)] There exists a non-negative stochastic process $\alpha_t\in M_G^1(0,T)$ such that for any 
        $(t,\omega,y,z)\in[0,T]\times\Omega\times\mathbb{R}\times\mathbb{R},$
        \begin{small}
            \begin{align*}
                |g(t,\omega,y,z)|\leq \alpha_t+\frac{L_2}{2}+L_1|y|+\frac{3L_2}{2}|z|^2;
            \end{align*}
        \end{small}
        \item[(iv)] The generator $g$ is uniformly continuous in $(t,\omega)$, i.e. there is a non-decreasing continuous function $w: [0, +\infty)\to [0, +\infty)$ such that $w(0)=0$ and
        \begin{small}
            \begin{align*}
                \sup_{y,z\in \mathbb{R}}|g(t,\omega,y,z)-g(t^\prime,\omega^\prime,y,z)|\leq w(|t-t^\prime|+\|\omega-\omega^\prime\|_\infty);
            \end{align*}
        \end{small}
        \item [(v)] For each $(t,\omega,y)\in [0,T]\times\Omega\times\mathbb{R}$, $g(t,\omega,y,\cdot)$
        is either convex or concave;
        \item [(vi)] Both the terminal value $\xi$ and $\small \int_{0}^{T}\alpha_tdt$ have exponential 
        moments of arbitrary order, i.e.
        \begin{small}
            \begin{align*}
                \widehat{\mathbb{E}}\left[\exp\left\{p|\xi|+p\int_{0}^{T}\alpha_tdt\right\}\right]<\infty,\ \text{for any } p\geq 1.
            \end{align*} 
        \end{small}
    \end{enumerate}
\end{assumption}

With the above assumptions, we give the main result of this subsection.
\begin{theorem}\label{thm-UBT}
    Assume that Assumption \ref{assumption-L-R} and \ref{assumption-UBT} hold.
    Then, the quadratic $G$-BSDE \eqref{DMFGBSDE-simplify} with double mean reflections
    admits a unique deterministic flat solution 
    $(Y,Z,K,A)\in \mathcal{E}_G(\mathbb{R})\times \mathcal{H}_G(\mathbb{R})\times \mathcal{L}_G(\mathbb{R})\times \mathcal{A}_D$. 
\end{theorem}

The following result is important for our subsequent proofs, 
which can be found in \cite{GuLinXu23}.
 \begin{lemma}[\cite{GuLinXu23}]\label{lemma-priori-estimate}
    Assume that $(Y,Z,K)\in \mathcal{E}_G^2(\mathbb{R})\times \mathcal{H}_G^2(\mathbb{R})\times \mathcal{L}_G^2(\mathbb{R})$
    is a solution to the following $G$-BSDE:
    \begin{small}
        \begin{align*}
            Y_t=\xi+(\bar{K}_T-\bar{K}_t)+\int_{t}^{T}g(s,Z_s)d\langle B\rangle_s
            -\int_{t}^{T}Z_sdB_s-(K_T-K_t),
        \end{align*}
    \end{small}where $\bar{K}\in S_G^p(0,T)$ is a non-increasing $G$-martingale and 
    $\alpha_t\in M_G^1(0,T)$ is a non negative stochastic process.
    Define $\tilde{\sigma}^2=(\bar{\sigma}/ \underline{\sigma})^2$ 
    and $\hat{\sigma}^2=\tilde{\sigma}^2(\bar{\sigma}^2\vee 1)$. Note that
    $\underline{\sigma}^2dt\leq d\langle B\rangle_t<\bar{\sigma}^2$.
    Suppose that there are two constants $p\geq 1$ and $\varepsilon>0$ such that
    \begin{small}
        \begin{align*}
            \widehat{\mathbb{E}}\left[\exp \left\{(2 p+\varepsilon) \gamma 
            \hat{\sigma}^2 \sup_{t \in[0, T]}\left|Y_t\right|+(2 p+\varepsilon) 
            \gamma \hat{\sigma}^2 \int_0^T \beta_t d t\right\}\right]<\infty .
        \end{align*}
    \end{small}
    
    Then, we have 
    \begin{enumerate}
        \item [(i)] If $|g(t,z)|\leq \alpha_t+\frac{\gamma}{2}|z|^2$, then, for each $t\in[0,T]$,
        \begin{small}
            \begin{align*}
                \exp \left\{p \gamma \tilde{\sigma}^2\left|Y_t\right|\right\} 
                \leq \widehat{\mathbb{E}}_t\left[\exp \left\{p \gamma 
                \hat{\sigma}^2|\xi|+p \gamma \hat{\sigma}^2 \int_t^T 
                \beta_s d s\right\}\right].
            \end{align*}
        \end{small}
        \item [(ii)] If $g(t,z)\leq \alpha_t+\frac{\gamma}{2}|z|^2$, then, for each $t\in[0,T]$,
        \begin{small}
            \begin{align*}
                \exp \left\{p \gamma \tilde{\sigma}^2Y_t^+\right\} 
                \leq \widehat{\mathbb{E}}_t\left[\exp \left\{p \gamma 
                \hat{\sigma}^2\xi^++p \gamma \hat{\sigma}^2 \int_t^T 
                \beta_s d s\right\}\right].
            \end{align*}
        \end{small}
    \end{enumerate}
\end{lemma}

For convenience, 
we use $C_i,i=1,2,...,12,$ to represent bounded constants 
that depend only on certain parameters in the subsequent proof.

\begin{theorem}\label{uniqueness-MRGBSDE}
    Assume that Assumption \ref{assumption-L-R} and \ref{assumption-UBT} hold.
    Then, the quadratic $G$-BSDE \eqref{DMFGBSDE-simplify} with double mean reflections
    admits most one deterministic flat solution 
    $(Y,Z,K,A)\in \mathcal{E}_G(\mathbb{R})\times \mathcal{H}_G(\mathbb{R})\times \mathcal{L}_G(\mathbb{R})\times \mathcal{A}_D$. 
\end{theorem}
\begin{proof} 
    It suffices to prove that the quadratic $G$-BSDE 
    with double mean reflections admits 
    most one solution on a sufficiently small interval $[0,T]$.
    For the general $T$, 
    the uniqueness on the whole interval follows from  
    the uniqueness on each small time interval.
    Let $(Y^i,Z^i,K^i,A^i)\in \mathcal{E}_G(\mathbb{R})\times \mathcal{H}_G(\mathbb{R})\times \mathcal{L}_G(\mathbb{R})\times \mathcal{A}_D,i=1,2,$
    be two deterministic flat solutions to the quadratic $G$-BSDE \eqref{DMFGBSDE-simplify}.
    By Theorem 3.9 in \cite{HTW22}, the following quadratic $G$-BSDE
    \begin{small}
        \begin{align*}
            y_t^i=\xi+\int_{t}^{T}g(s,\omega_{. \wedge s},Y_s^i,z_s^i)
            d\langle B\rangle_s-\int_{t}^{T}z_s^idB_s-(k_T^i-k_t^i),\ \forall t\in[0,T]
        \end{align*}
    \end{small}admits a unique solution.

    We only consider the case when $g$ is convex in $z$. When 
    $g$ is concave in $z$, we use $\theta l^1-l^2$ and $\theta l^2-l^1$
    in the subsequent definition of $\delta_\theta l$ and $\delta_\theta \tilde{l}$, respectively.
    For each $\theta\in (0,1)$, we set
    \begin{small}
        \begin{align*}
            \delta_\theta l:=\frac{l^1-\theta l^2}{1-\theta}, 
            \ \delta_\theta\tilde{l}:=\frac{l^2-\theta l^1}{1-\theta},\ \text{and}
            \ \delta_\theta\bar{l}:=|\delta_\theta l|+|\delta_\theta\tilde{l}|,
        \end{align*}
    \end{small}for $l=y,Y,z$. Hence, the triplet $(\delta_\theta y,\delta_\theta z,\frac{1}{1-\theta}k^1)$
    solves the following $G$-BSDE on $[0,T]$:
    \begin{small}
        \begin{align*}
            \delta_\theta y_t=&\xi+\frac{\theta}{1-\theta}(k_T^2-k_t^2)+\int_{t}^{T}\delta_\theta g 
            (s,\delta_\theta Y_s,\delta_\theta z_s)d\langle B\rangle_s\\ 
            &-\int_{t}^{T}\delta_\theta z_sdB_s-\frac{1}{1-\theta}(k_T^1-k_t^1),
        \end{align*}
    \end{small}with $\small \delta_\theta g(t,\hat{y},\hat{z})=\frac{1}{1-\theta}(g(t,(1-\theta )\hat{y}+\theta Y_t^2,(1-\theta )\hat{z}+\theta z_t^2)-\theta g(t,Y_t^2,z_t^2)).$
    By $(iii)$ of Assumption \ref{assumption-UBT}, we have 
    \begin{small}
        \begin{align*}
            |g(t,\omega,y,z)|\leq \alpha_t+\frac{L_2}{2}+L_1|y|+\frac{3L_2}{2}|z|^2,
        \end{align*}
    \end{small}with convexity of $g$ in $z$, we get 
    \begin{small}
        \begin{align*}
            \delta_\theta g & \left(t, \delta_\theta Y, \hat{z}\right) \\
            & =\frac{1}{1-\theta}\left(g\left(t,(1-\theta) \delta_\theta Y+\theta Y_t^2,(1-\theta) \hat{z}+\theta z_t^2\right)-\theta g\left(t, Y_t^2, z_t^2\right)\right) \\
            & \leq L_1\left|\delta_\theta Y\right|+L_1\left|Y_t^2\right|+\frac{1}{1-\theta}\left(g\left(t, Y_t^2,(1-\theta) \hat{z}+\theta z_t^2\right)-\theta g\left(t, Y_t^2, z_t^2\right)\right) \\
            & \leq L_1\left|\delta_\theta Y\right|+L_1\left|Y_t^2\right|+g\left(t, Y_t^2, \hat{z}\right) \leq \alpha_t+\frac{L_2}{2}+2 L_1\left|Y_t^2\right|+L_1\left|\delta_\theta Y\right|+\frac{3 L_2}{2}|\hat{z}|^2
        \end{align*}
    \end{small}Using assertion $(ii)$ of Lemma \ref{lemma-priori-estimate}, taking 
    $\beta_t=\alpha_t+\frac{L_2}{2}+2 L_1\left|Y_t^2\right|+L_1\left|\delta_\theta Y\right|$ and $\gamma=3 L_2$,
    we derive that 
    \begin{small}
        \begin{align*}
            \exp  \left\{3 p L_2 \tilde{\sigma}^2\left(\delta_\theta y_t\right)^{+}\right\}
            \leq \widehat{\mathbb{E}}_t\left[\exp \left\{3 p L_2 \hat{\sigma}^2
            \left(|\xi|+\chi+L_1T\left(\sup_{t \in[0, T]}\left|
                \delta_\theta Y_t\right|\right)\right)\right\}\right] , 
    \end{align*}
    \end{small}where $\small \chi=\int_0^T \alpha_s d s+\frac{L_2 T}{2}+2 L_1 T\left(\sup_{t \in[0, T]}\left|Y_t^1\right|+\sup_{t \in[0, T]}\left|Y_t^2\right|\right)$.
    Similarly, we have
    \begin{small}
        \begin{align*}
            \exp  \left\{3 p L_2 \tilde{\sigma}^2\left(\delta_\theta \tilde{y}_t\right)^{+}\right\}
            \leq \widehat{\mathbb{E}}_t\left[\exp \left\{3 p L_2 \hat{\sigma}^2
            \left(|\xi|+\chi+L_1T\left(\sup_{t \in[0, T]}\left|
                \delta_\theta \tilde{Y}_t\right|\right)\right)\right\}\right]. 
    \end{align*}
    \end{small}Note that
    \begin{small}
        \begin{align*}
            \left(\delta_\theta y\right)^{-} \leq\left(\delta_\theta \tilde{y}\right)^{+}+2\left|y^2\right| \text { and }\left(\delta_\theta \tilde{y}\right)^{-} \leq\left(\delta_\theta y\right)^{+}+2\left|y^1\right|.
        \end{align*}
    \end{small}Therefore, we have 
    \begin{small}
        \begin{align*}
            \exp & \left\{3 p L_2 \tilde{\sigma}^2\left|\delta_\theta y_t\right|\right\} \vee \exp \left\{3 p L_2 \tilde{\sigma}^2\left|\delta_\theta \tilde{y}_t\right|\right\} \\
            & \leq \exp \left\{3 p L_2 \hat{\sigma}^2\left(\left(\delta_\theta y_t\right)^{+}+\left(\delta_\theta \tilde{y}_t\right)^{+}+2\left|y_t^1\right|+2\left|y_t^2\right|\right)\right\} \\
            & \leq \widehat{\mathbb{E}}_t\left[\exp \left\{3 p L_2 \hat{\sigma}^2\left(|\xi|+\tilde{\chi}+L_1T\left(\sup_{t \in[0, T]} \delta_\theta \bar{Y}_t\right)\right)\right\}\right]^2,
        \end{align*}
    \end{small}where $\small \tilde{\chi}=  \int_0^T \alpha_t d t+\frac{L_2 T}{2}+2 L_1 T\left(\sup_{t \in[0, T]}\left|Y_t^1\right|+\sup_{t \in[0, T]}\left|Y_t^2\right|\right)+\sup_{t \in[0, T]}\left|\bar{Y}_t^1\right|+\sup_{t \in[0, T]}\left|\bar{Y}_t^2\right|.$
    With the help of Theorem \ref{Doob-inequality}, Remark \ref{Doob-exp} and H\"{o}lder's inequality under $G$-framework, we get that for each
    $p\geq 1$ and $t\in[0,T]$,
    \begin{small}
        \begin{align}\label{delta-bar-y}
            & \widehat{\mathbb{E}}\left[\exp \left\{3 p L_2 \tilde{\sigma}^2 \sup_{t \in[0, T]} \delta_\theta \bar{y}_t\right\}\right] \notag\\
            & \quad \leq \widehat{\mathbb{E}}\left[\exp \left\{3 p L_2 \tilde{\sigma}^2 \sup_{t \in[0, T]}\left|\delta_\theta y_t\right|\right\} \exp \left\{3 p L_2 \tilde{\sigma}^2 \sup_{t \in[0, T]}\left|\delta_\theta \tilde{y}_t\right|\right\}\right]\notag\\
            & \quad \leq \hat{A}_G \widehat{\mathbb{E}}\left[\exp \left\{24 p L_2 \hat{\sigma}^2\left(|\xi|+\widetilde{\chi}+L_1T\left(\sup_{t \in[0, T]} \delta_\theta \bar{Y}_t\right)\right)\right\}\right] . 
        \end{align}
    \end{small}

    Set 
    \begin{small}
        \begin{align*}
            C_1:=\sup_{t \in[0, T]}|\bar{A}^0_{t}|+\frac{6C}{c} \sup_{t \in[0, T]} 
        \widehat{\mathbb{E}}\left[\sum_{i=1}^2|y_t^i|\right]+\frac{C}{c}\widehat{\mathbb{E}}\left[|\xi|\right],
        \end{align*}
    \end{small}where $\bar{A}^0$ is the second component of the solution to the Skorokhod problem 
    $\mathbb{SP}_{\bar{l}^0}^{\bar{r}^0}(\bar{s}^0)$ and 
    \begin{small}
        \begin{align*}
            \bar{l}^0(t,x)=l^0(T-t,x)=\widehat{\mathbb{E}}\left[L(T-t,x)\right], 
            \bar{r}^0(t,x)=r^0(T-t,x)=\widehat{\mathbb{E}}\left[R(T-t,x)\right], 
            \bar{s}^0=\widehat{\mathbb{E}}[\xi].
        \end{align*}
    \end{small}Note that \eqref{slr}, for simplicity, denote $s^i=s^{Y^i},l^i=l^{Y^i},r^i=r^{Y^i},i=1,2$ and 
    $a=\widehat{\mathbb{E}}\left[|\xi|\right]$. Define $\bar{s}^i,\bar{l}^i,\bar{r}^i,i=1,2,$
    as in \eqref{bar-slr}. Let $\bar{A}^i$ be the second component of the solution to 
    the Skorokhod problem $\mathbb{SP}_{\bar{l}^i}^{\bar{r}^i}(\bar{s}^i)$, $i=1,2$. 
    Then, we have
    \begin{small}
        \begin{align*}
            Y_t^i=y_t^i+\bar{A}^{i}_{T-t}, \quad \forall t \in[0, T].
        \end{align*}
    \end{small}Hence, we derive that  
    \begin{small}
        \begin{align}\label{delta-Y}
            \left|\delta_{\theta}Y_{t}\right|&=
            \left|\frac{\left(y_{t}^{1}+\bar{A}^{1}_{T-t}\right)-
            \theta\left(y_{t}^{2}+\bar{A}^{2}_{T-t}\right)}{1-\theta}\right|\notag\\
            &\leq \left|\delta_{\theta}y_{t}\right|+\frac{\theta}{1-\theta}
            \sup_{t\in[0,T]}\left|\bar{A}^{1}_{t}-\bar{A}^{2}_{t}\right|
            +\sup_{t\in[0,T]}\left|\bar{A}^{1}_{t}-\bar{A}^{0}_{t}\right|+
            \sup_{t\in[0,T]}\left|\bar{A}^{0}_{t}\right|.
        \end{align}
    \end{small}On the one hand, by Theorem \ref{thm-SP-K1-K2}, we have
    \begin{small}
        \begin{equation*}%\label{re-6}
            \sup_{t\in[0,T]}\left|\bar{A}^{1}_{t}-\bar{A}^{0}_{t}\right|\leq\frac{C}{c}\sup_{t\in[0,T]}\left|\bar{s}^{1}_{t}-\bar{s}^{0}_{t}\right|+\frac{1}{c}\left(\bar{L}^0_T\vee \bar{R}^0_T\right),%\leq \frac{3C}{c}\sup_{t\in[0,T]}\mathbf{E}\left[\left|y_t^2\right| \right]+\frac{C}{c}\mathbf{E}[\left|\xi\right|].   
        \end{equation*}
    \end{small}where 
    \begin{small}
        \begin{align*}
            \bar{L}_T^0=\sup_{(t,x)\in[0,T]\times \mathbb{R}}|\bar{l}^1(t,x)-\bar{l}^0(t,x)|, \ %=\sup_{(t,x)\in[0,T]\times \mathbb{R}}|l^2(t,x)-l^0(t,x)|,\\
            \bar{R}_T^0=\sup_{(t,x)\in[0,T]\times \mathbb{R}}|\bar{r}^1(t,x)-\bar{r}^0(t,x)|.%=\sup_{(t,x)\in[0,T]\times \mathbb{R}}|r^2(t,x)-r^0(t,x)|.
         \end{align*}
    \end{small}Moreover, according to the definition of $\bar{s}^i$, $\bar{l}^i$, $\bar{r}^i$ for $i=0,1$, we obtain that
    \begin{small}
        \begin{align}\label{A1-A0}
            \sup_{t\in[0,T]}\left|\bar{A}^{1}_{t}-\bar{A}^{0}_{t}\right| \leq 
            \frac{3C}{c}\sup_{t\in[0,T]}\widehat{\mathbb{E}}\left[\left|y_t^1\right| \right]
            +\frac{C}{c}\widehat{\mathbb{E}}[\left|\xi\right|].
         \end{align}
    \end{small}On the other hand, using similar method as \eqref{hat-bar-A}, we get
    \begin{small}
        \begin{align}\label{A1-A2}
            \frac{\theta}{1-\theta}\sup_{t\in[0,T]}\left|\bar{A}^{1}_{t}-\bar{A}^{2}_{t}\right|
            &\leq \frac{\theta}{1-\theta}\frac{3C}{c}\sup_{t \in[0, T]}\widehat{\mathbb{E}}[|y_t^1-y_t^2|]\notag\\
            &\leq \frac{3C}{c}\sup_{t\in[0,T]}\widehat{\mathbb{E}} \left[\left|\frac{ y_t^1-\theta y_t^2+(1-\theta) y_t^1 }{1-\theta}\right|\right]\notag\\
            &\leq \frac{3C}{c}\sup_{t\in[0,T]}\widehat{\mathbb{E}}\left[|\delta_{\theta}y_{t}|\right]+\frac{3C}{c}\sup_{t\in[0,T]}\widehat{\mathbb{E}}\left[|y_t^1|\right].   
        \end{align}
    \end{small}Thus, take \eqref{A1-A0} and \eqref{A1-A2} into \eqref{delta-Y}, we obtain 
    \begin{small}
        \begin{align}\label{delta-theta-Y}
            \left|\delta_{\theta}Y_{t}\right|\leq C_1+|\delta_{\theta}y_{t}|+\frac{3C}{c}\sup_{t\in[0,T]}\widehat{\mathbb{E}}\left[|\delta_{\theta}y_{t}|\right].
        \end{align}
    \end{small}Similarly, we have
    \begin{small}
        \begin{align*}
            \left|\delta_{\theta}\tilde{Y}_{t}\right|\leq C_1+|\delta_{\theta}\tilde{y}_{t}|+\frac{3C}{c}\sup_{t\in[0,T]}\widehat{\mathbb{E}}\left[|\delta_{\theta}\tilde{y}_{t}|\right].
        \end{align*}
    \end{small}It follows from the above analysis that 
    \begin{small}
        \begin{align*}
            \delta_{\theta}\bar{Y}_{t}\leq 2C_1+\delta_{\theta}\bar{y}_{t}+\frac{6C}{c}
            \widehat{\mathbb{E}}\left[\sup_{t\in[0,T]}\delta_{\theta}\bar{y}_{t}\right].
        \end{align*}
    \end{small}Hence, we deduce that 
    \begin{small}
        \begin{align*}
            \widehat{\mathbb{E}} & {\left[\exp \left\{3 p L_2 \tilde{\sigma}^2 \sup_{t \in[0, T]} \delta_\theta \bar{Y}_t\right\}\right] } \\
            & \leq e^{6 p L_2 \tilde{\sigma}^2 C_1 }
            \widehat{\mathbb{E}}\left[\exp \left\{3 p L_2 \tilde{\sigma}^2 
            \sup_{t \in[0, T]} \delta_\theta \bar{y}_t\right\}\right] 
            \widehat{\mathbb{E}}\left[\exp \left\{18\frac{Cp L_2 \tilde{\sigma}^2}{c}  
            \sup_{t \in[0, T]} \delta_\theta \bar{y}_t\right\}\right] \\
            & \leq e^{6 p L_2 \tilde{\sigma}^2 C_1 }\widehat{\mathbb{E}}
            \left[\exp \left\{(6+36\frac{C}{c}) p L_2 \tilde{\sigma}^2 
            \sup_{t \in[0, T]} \delta_\theta \bar{y}_t\right\}\right] \\
            & \leq \hat{A}_G \widehat{\mathbb{E}}\left[\exp \left\{(48+288\frac{C}{c}) 
            p L_2 \hat{\sigma}^2\left(|\xi|+\widetilde{\chi}+C_1+L_1T
            \left(\sup_{t \in[t, T]} \delta_\theta \bar{Y}_t\right)\right)\right\}\right],
        \end{align*}
    \end{small}in which we have used \eqref{delta-bar-y} in the third inequality.

    Therefore, when $T$ is small enough such that $(32+192\frac{C}{c}) L_1T \hat{\sigma}^2 / \tilde{\sigma}^2<1$, 
    by H\"{o}lder's inequality under $G$-framework, we deduce that for any $p\geq1,$
    \begin{small}
        \begin{align*}
            \widehat{\mathbb{E}} & {\left[\exp \left\{3 p L_2 \tilde{\sigma}^2 \sup_{t \in[0, T]} \delta_\theta \bar{Y}_t\right\}\right] } \\
            \leq & \hat{A}_G \widehat{\mathbb{E}}\left[\exp \left\{(96+576\frac{C}{c}) 
            p L_2 \hat{\sigma}^2\left(|\xi|+\widetilde{\chi}+C_1\right)\right\}\right]^{\frac{1}{2}} \\
            & \times\left(\widehat{\mathbb{E}}\left[\exp \left\{(96+576\frac{C}{c}) 
            L_1T p L_2 \hat{\sigma}^2 \sup_{t \in[0, T]} \delta_\theta \bar{Y}_t\right\}\right]\right)^{\frac{1}{2}} \\
            \leq & \hat{A}_G \widehat{\mathbb{E}}\left[\exp \left\{(96+576\frac{C}{c}) 
            p L_2 \hat{\sigma}^2\left(|\xi|+\widetilde{\chi}+C_1\right)\right\}\right] \\
            & \times \widehat{\mathbb{E}}\left[\exp \left\{3 p L_2 \tilde{\sigma}^2 \sup_{t \in[0, T]} \delta_\theta \bar{Y}_t\right\}
            \right]^{(32+192\frac{C}{c}) L_1T \hat{\sigma}^2 / \tilde{\sigma}^2}.       
        \end{align*}
    \end{small}and the following holds for any $p\geq1$ and $\theta\in(0,1)$:
    \begin{small}
        \begin{align*}
            & \widehat{\mathbb{E}}\left[\exp \left\{3 p L_2 \tilde{\sigma}^2 \sup_{t \in[0, T]} 
            \delta_\theta \bar{Y}_t\right\}\right] \\
            & \quad \leq \widehat{\mathbb{E}}\left[\hat{A}_G \exp 
            \left\{(96+576\frac{C}{c}) p L_2 \hat{\sigma}^2\left(|\xi|+\widetilde{\chi}+C_1\right)
            \right\}\right]^{\frac{1}{1-(32+192\frac{C}{c}) L_1T \hat{\sigma}^2 / \tilde{\sigma}^2}}<\infty.
        \end{align*}
    \end{small}Note that $Y^1-Y^2=(1-\theta)(\delta_\theta Y-Y^2)$. It follows that 
    \begin{small}
        \begin{align*}
            & 3 \tilde{\sigma}^2 \widehat{\mathbb{E}}\left[\sup_{t \in[0, T]}\left|Y_t^1-Y_t^2\right|\right] \\
            & \quad \leq(1-\theta)\left(\frac{1}{L_2} \sup_{\theta \in(0,1)} 
            \widehat{\mathbb{E}}\left[\exp \left\{3 L_2 \tilde{\sigma}^2 
            \sup_{t \in[0, T]} \delta_\theta \bar{Y}_t\right\}\right]+
            3 \tilde{\sigma}^2 \widehat{\mathbb{E}}
            \left[\sup_{t \in[0, T]}\left|Y_t^2\right|\right]\right).
        \end{align*}
    \end{small}Letting $\theta \to 1$ yields that $Y^1=Y^2$ on small time interval $[0,T]$. 
    Then, applying $G$-It\^{o}'s formula to $|Y^1-Y^2|^2$, we have 
    $Z^1=Z^2,K^1=K^2$ and $A^1=A^2$ on small interval $[0,T]$. 

    The proof is complete.
\end{proof}

\begin{proof}[{\bf Proof of Theorem \ref{thm-UBT}}] 
    Similar to \cite{HTW22}, we will construct a solution through a sequence of
    quadratic $G$-BSDEs with bounded terminal value and generator.
    Denote by $g_0(t)=g(t,0,0)$ for convenience. Then for each positive integer
    $m\geq1$, set 
    \begin{small}
        \begin{align*}
            &l^{(m)}=(l\wedge m)\vee(-m) \text{ for } l=\xi, g_0(t),\\
            &g^{(m)}(t,y,z)=g(t,y,z)-g_0(t)+g_0^{(m)}(t).
        \end{align*}
    \end{small}

    It suffices to prove that the quadratic $G$-BSDE 
    with double mean reflections on a small interval $[t_0,t_0+h]$ admits 
    a solution. For the whole interval $[0,T]$, 
    similar to the analysis of Step 2 in the proof of the Theorem \ref{thm-BT}, 
    we can obtain existence on the 
    whole time interval. Therefore, we only deal with the case 
    on a small interval. The proof will be divided into the following parts.
    
    {\bf Step 1.} To begin, we first provide a representation 
    for $Y^{(m)}$ for $m\geq 1$. For any given $t_0\in[0,T]$, and let $h\in(0,T-t_0]$. Denote 
    $\xi^{(m)}=Y_{t_0+h}^{(m)}$, 
    then by Theorem \ref{thm-BT}, we know that 
    $\left(Y^{(m)}, Z^{(m)}, K^{(m)},A^{(m)}\right) \in \mathcal{E}_G(\mathbb{R})\times \mathcal{H}_G(\mathbb{R})\times \mathcal{L}_G(\mathbb{R})\times \mathcal{A}_D$ 
    is a unique solution to the following quadratic $G$-BSDE with double mean reflections 
    with data $(\xi^{(m)},g^{(m)})$ on $[t_0,t_0+h]$:
    \begin{small}
        \begin{equation}\label{Ytm}
            \begin{cases}
                Y_t^{(m)}=\xi^{(m)}+\int_t^{t_0+h} g^{(m)}\left(s, Y_s^{(m)}, Z^{(m)}_s\right) 
                d\langle B\rangle_s-\int_t^{t_0+h} Z_s^{(m)} dB_s -
                \left(K_{T}^{(m)}-K_t^{(m)}\right)+\left(A^{(m)}_T-A^{(m)}_t\right), \\
                \widehat{\mathbb{E}}[L(t,Y^{(m)}_t)]\leq 0 \leq 
                \widehat{\mathbb{E}}[R(t,Y^{(m)}_t)], \ t_0 \leq t \leq t_0+h,\\ 
                A_t^{(m)}=(A_t^R)^{(m)}-(A_t^L)^{(m)},\ 
                (A_t^R)^{(m)},(A_t^L)^{(m)}\in I[t_{0},t_0+h],  \\
                \int_{t_0}^{t_0+h}\widehat{\mathbb{E}}[R(t,Y_t^{(m)})]d(A_t^R)^{(m)}=
                \int_{t_0}^{t_0+h}\widehat{\mathbb{E}}[L(t,Y_t^{(m)})]d(A_t^L)^{(m)}=0.
            \end{cases}
        \end{equation}
    \end{small}
    
    Let $(y^{(m)},z^{(m)},k^{(m)})$ be the solution to the following quadratic $G$-BSDE
    \begin{small}
        \begin{align}\label{y-m}
            y_t^{(m)}=\xi^{(m)}+\int_t^{t_0+h} g^{(m)}\left(s, Y_s^{(m)}, z^{(m)}_s\right) 
                d\langle B\rangle_s-\int_t^{t_0+h} z_s^{(m)} dB_s -
                \left(k_{T}^{(m)}-k_t^{(m)}\right).
        \end{align}
    \end{small}Note that $A^{(m)}$ is a deterministic function, by Lemma B.1 in \cite{HTW22}, we 
    know that $ (Y_{.}^{(m)}-(A^{(m)}_{t_0+h}-A^{(m)}_{.}),Z_{.}^{(m)},K_{.}^{(m)})$ 
    is also a solution to $G$-BSDE \eqref{y-m} on $[t_0,t_0+h]$, which yields that 
    \begin{small}
        \begin{equation}\label{Y-y-relationship}
            \left(Y^{(m)}_t,Z^{(m)}_t,K^{(m)}_t\right)=
            \left(y^{(m)}_t+A^{(m)}_{t_0+h}-A^{(m)}_t,z^{(m)}_t,k^{(m)}_t\right),\quad \forall t\in [t_0,t_0+h].
        \end{equation}
    \end{small}

    For any $(t,x)\in [t_0,t_0+h]\times\mathbb{R}$, we define
    \begin{small}
        \begin{equation}\label{slr-m}
            \begin{split}
                &s^{(m)}_t:=\widehat{\mathbb{E}}\left[\xi^{(m)}+
                \int_{t_0}^{t_0+h} g^{(m)}(s,Y_s^{(m)},z^{(m)}_s) d\langle B\rangle_s\right]
                -\widehat{\mathbb{E}}\left[\xi^{(m)}+\int_t^{t_0+h} 
                g^{(m)}(s,Y_s^{(m)},z^{(m)}_s) d\langle B\rangle_s\right]\\
                &\qquad=\widehat{\mathbb{E}}\left[y^{(m)}_{t_0}\right]-\widehat{\mathbb{E}}\left[y^{(m)}_t\right], 
                \ \ \ a:=\widehat{\mathbb{E}}\left[\xi^{(m)}\right]\\
                &l^{(m)}(t,x):=\widehat{\mathbb{E}}\left[L(t,y^{(m)}_t-\widehat{\mathbb{E}}\left[y^{(m)}_t\right]+x)\right], 
                r^{(m)}(t,x):=\widehat{\mathbb{E}}\left[R(t,y^{(m)}_t-\widehat{\mathbb{E}}\left[y^{(m)}_t\right]+x)\right].  
            \end{split}
        \end{equation}
    \end{small}We can check that $l^{(m)},r^{(m)}$ satisfy Assumption \ref{assumption-l-r} and the following holds:
    \begin{small}
        \begin{align*}
            l^{(m)}(t_0+h,a)=\widehat{\mathbb{E}}[L(t_0+h,y^{(m)}_{t_0+h})]\leq 0\leq 
            \widehat{\mathbb{E}}[R(t_0+h,y^{(m)}_{t_0+h})]=r^{(m)}(t_0+h,a).
        \end{align*}
    \end{small}Hence, Theorem \ref{thm-BSP} ensures that the backward 
    Skorokhod problem $\mathbb{BSP}_{l^{(m)}}^{r^{(m)}}(s^{(m)},a)$ 
    admits a unique solution $(x^{(m)},A^{(m)})$. 
    Recalling that for any $t\in[t_0,t_0+h]$, 
    $A^{(m)}_t=\bar{A}^{(m)}_h-\bar{A}^{(m)}_{t_0+h-t}$,  
    where $\bar{A}^{(m)}$ is the second component of the 
    solution to the Skorokhod problem 
    $\mathbb{SP}_{\bar{l}^{(m)}}^{\bar{r}^{(m)}}(\bar{s}^{(m)})$,  
    where 
    \begin{small}
        \begin{equation}\label{bar-slr-m}
            \begin{split}
            &\bar{s}^{(m)}_{t}=a+s^{(m)}_{t_0+h}-s^{(m)}_{t_0+h-t}=\widehat{\mathbb{E}}[y^{(m)}_{t_0+h-t}],\\ 
            &\bar{l}^{(m)}(t,x)=l^{(m)} (t_0+h-t,x), \ \bar{r}^{(m)}(t,x)=r^{(m)}(t_0+h-t,x).
            \end{split}
        \end{equation} 
    \end{small}Therefore, we obtain the following representation for $Y^{(m)}$,
    \begin{small}
        \begin{equation}\label{representation-Y-m}
            Y^{(m)}_t=y^{(m)}_t+A^{(m)}_{t_0+h}-A^{(m)}_t=y^{(m)}_t+\bar{A}^{(m)}_{t_0+h-t}, 
            \quad \forall t\in[t_0,t_0+h].   
        \end{equation}
    \end{small}

    {\bf Step 2. The uniform estimates.} 
    For any $p\geq1$, we claim that the following two estimates hold: 
    \begin{small}
        \begin{align}\label{estimate-y-m}
            \sup_{m}\widehat{\mathbb{E}}\left[\sup_{t\in [t_0,t_0+h]}
            \exp\left\{3pL_2\tilde{\sigma}^2\left|y_t^{(m)}\right|\right\}\right] 
            <\infty,
        \end{align}
    \end{small}and
    \begin{small}
        \begin{align}\label{estimate-Y-m}
            \sup_{m}\widehat{\mathbb{E}}\left[\sup_{t\in [t_0,t_0+h]}
            \exp\left\{3pL_2\tilde{\sigma}^2\left|Y_t^{(m)}\right|\right\}+\left(\int_{t_0}^{t_0+h}\left|Z_t^{(m)}\right|^2dt\right)^n
            +\left|K_{t_0+h}^{(m)}\right|^n+\left|A_{t_0+h}^{(m)}\right|\right]<\infty,\ \forall n\geq1.
        \end{align}
    \end{small}
    
    The proofs of inequalities \eqref{estimate-y-m} and \eqref{estimate-Y-m} 
    are provided in the Appendix.
    
    {\bf Step 3. $\mathbf{\theta}$-method.} Similar to Theorem \ref{uniqueness-MRGBSDE}, 
    we only consider the case when $g$ is convex in $z$.
    For each $\theta\in (0,1)$, we set
    \begin{small}
        \begin{align*}
            &\delta_\theta l^{(m,q)}:=\frac{l^{(m,q)}-\theta l^{(m)}}{1-\theta}, 
            \ \delta_\theta\tilde{l}:=\frac{l^{(m)}-\theta l^{(m,q)}}{1-\theta},\\
            &\qquad \delta_\theta\bar{l}:=|\delta_\theta l^{(m,q)}|+|\delta_\theta\tilde{l}^{(m,q)}|,
        \end{align*}
    \end{small}with $l=y,z,Y$. Then, we can deduce that 
    $(\delta_\theta y^{(m,q)},\delta_\theta z^{(m,q)},\frac{1}{1-\theta}k^{(m,q)})$ admits 
    the following $G$-BSDE:
    \begin{small}
        \begin{align*}
            \delta_\theta y_t^{(m, q)}= & \delta_\theta \xi^{(m, q)}+
            \frac{\theta}{1-\theta}\left(k_{t_0+h}^{(m)}-k_t^{(m)}\right)-
            \int_t^{t_0+h} \delta_\theta z_s^{(m, q)} d B_s 
            -\frac{1}{1-\theta}\left(k_{t_0+h}^{(m+q)}-k_t^{(m+q)}\right)\\
            &+\int_t^{t_0+h}\left(\delta_\theta g^{(m, q)}\left(s, \delta_\theta 
            Y_s^{(m, q)}, \delta_\theta z_s^{(m, q)}\right) 
             +\delta_\theta g_0^{(m, q)}(s)\right) d\langle B\rangle_s,
        \end{align*}
    \end{small}where 
    \begin{small}
        \begin{align*}
            \delta_\theta \xi^{(m, q)}= & \frac{\xi^{(m+q)}-\theta \xi^{(m)}}{1-\theta}, \\
            \delta_\theta g_0^{(m, q)}(t)= & \frac{1}{1-\theta}\left(g_0^{(m+q)}(t)-\theta g_0^{(m)}(t)\right)-g_0(t), \\
            \delta_\theta g^{(m, q)}(t, y, z)= & \frac{1}{1-\theta}\left(g\left(t,(1-\theta) y+\theta Y_t^{(m)},(1-\theta) z+\theta z_t^{(m)}\right)\right. \\
            & \left.-\theta g\left(t, Y_t^{(m)}, z_t^{(m)}\right)\right) .
        \end{align*}
    \end{small}By simple derivation, we have 
    \begin{small}
        \begin{align*}
            &\delta_\theta \xi^{(m, q)}  \leq|\xi|+\frac{1}{1-\theta}(|\xi|-m)^{+}, \ \ \ 
            \delta_\theta g_0^{(m, q)}(t)  \leq \frac{2}{1-\theta}\left(\left|g_0(t)\right|-m\right)^{+} ,\\
            &\delta_\theta g^{(m, q)}(t, \delta_\theta Y^{(m, q)}, z)\leq \alpha_t+\frac{L_2}{2}
            +2L_1|Y_t^{(m)}|+L_1|\delta_\theta Y^{(m, q)}|+\frac{3L_2}{2}|z|^2.
        \end{align*}
    \end{small}Denote
    \begin{align*}
        \chi^{(m, q)} & =\int_t^{t_0+h} \alpha_s d s+\frac{L_2 h}{2}+
        2 L_1 h\left(\sup_{t \in[t_0, t_0+h]}\left|Y_t^{(m)}\right|+\sup_{t \in[t_0, t_0+h]}\left|Y_t^{(m+q)}\right|\right) \\
        \zeta^{(m, q)} & =|\xi|+\frac{L_2 h}{2}+\int_t^{t_0+h} \alpha_s d s+L_1 h\left(\sup_{t \in[t_0, t_0+h]}\left|Y_t^{(m)}\right|+\sup_{t \in[t_0, t_0+h]}\left|Y_t^{(m+q)}\right|\right),\\
        \rho(\theta,m)&=\frac{1}{1-\theta}(|\xi|-m)^{+}+\frac{2}{1-\theta}\int_{t}^{t_0+h}\left(\left|g_0(s)\right|-m\right)^{+}ds.
    \end{align*}

    Taking 
    $\beta_t=\alpha_t+\frac{L_2}{2}+2L_1|Y_t^{(m)}|+L_1|\delta_\theta Y_t^{(m, q)}|$ 
    and $\gamma=3L_2$ in assertion $(ii)$ of Lemma \ref{lemma-priori-estimate}, we derive that
    \begin{small}
        \begin{align*}
            &\exp \left\{3 p L_2  \tilde{\sigma}^2\left(\delta_\theta y_t^{(m, q)}
            \right)^{+}\right\} \\
            \leq & \widehat{\mathbb{E}}_t\left[\exp \left\{3 p L_2 
            \hat { \sigma } ^ { 2 } \left(|\xi|+\rho(\theta, m)+\frac{L_2 h}{2}
            \right.\right.\right. 
             \left.\left.\left.+\int_t^{t_0+h}\left(\alpha_s+
            2 L_1\left|Y_s^{(m)}\right|+L_1\left| \delta_\theta Y_s^{(m, q)} 
            \right|\right) d s\right)\right\}\right] \\
            \leq & \widehat{\mathbb{E}}_t\left[\exp \left\{3 p L_2
            \hat { \sigma } ^ { 2 } \left(|\xi|+\rho(\theta, m)+\chi^{(m, q)}\right.\right.\right.
             \left.\left.\left.+L_1h\left(\sup_{t \in[t_0, t_0+h]}\left|
                \delta_\theta Y_t^{(m, q)}\right|\right)\right)\right\}\right],
        \end{align*}
    \end{small}Similarly, we get 
    \begin{small}
        \begin{align*}
            &\exp \left\{3 p L_2  \tilde{\sigma}^2\left(\delta_\theta \tilde{y}_t^{(m, q)}
            \right)^{+}\right\}\\
            &\leq \widehat{\mathbb{E}}_t\left[\exp \left\{3 p L_2
            \hat { \sigma } ^ { 2 } \left(|\xi|+\rho(\theta, m)+
            \chi^{(m, q)}
            +L_1h\left(\sup_{t \in[t_0, t_0+h]}\left|
            \delta_\theta \tilde{Y}_t^{(m, q)}\right|\right)\right)\right\}\right].
        \end{align*}
    \end{small}Then, by the following facts
    \begin{small}
        \begin{align*}
            \left(\delta_\theta y^{(m, q)}\right)^{-} \leq\left(\delta_\theta \tilde{y}^{(m, q)}\right)^{+}
            +2\left|y^{(m, q)}\right| \text { and }\left(\delta_\theta \tilde{y}^{(m, q)}\right)^{-} 
            \leq\left(\delta_\theta y^{(m, q)}\right)^{+}+2\left|y^{(m)}\right|,
        \end{align*}
    \end{small}we conclude that 
    \begin{small}
        \begin{align*}
            &\exp \left\{ 3 p L_2 \tilde{\sigma}^2\left|\delta_\theta y_t^{(m, q)}
            \right|\right\} \vee \exp \left\{3 p L_2 \tilde{\sigma}^2
            \left|\delta_\theta \tilde{y}_t^{(m, q)}\right|\right\} \\
            \leq & \exp \left\{3 p L_2 \tilde{\sigma}^2\left(
                \left(\delta_\theta y_t^{(m, q)}\right)^{+}+\left(
                    \delta_\theta \tilde{y}_t^{(m, q)}\right)^{+}+
                    2\left|y_t^{(m)}\right|+2\left|y_t^{(m+q)}\right|\right)\right\} \\
            \leq & \left(\widehat{\mathbb{E}}_t\left[\exp \left\{3 p L_2 
            \hat{\sigma}^2\left(|\xi|+\rho(\theta, m)+\chi^{(m, q)}+L_1h
            \left(\sup_{t \in[t_0, t_0+h]}\delta_\theta \bar{Y}_t^{(m, q)}\right)\right)\right\}\right]\right)^2 \\
            & \times \exp \left\{6 p L_2 \tilde{\sigma}^2
            \left(\left|y_t^{(m)}\right|+\left|y_t^{(m+q)}\right|\right)\right\} \\
            \leq & \widehat{\mathbb{E}}_t\left[\exp \left\{3 p L_2 \hat{\sigma}^2
            \left(|\xi|+\rho(\theta, m)+\chi^{(m, q)}+L_1h\left(\sup_{t \in[t_0, t_0+h]} 
            \delta_\theta \bar{Y}_t^{(m, q)}\right)\right)\right\}\right]^2 \\
            & \times \widehat{\mathbb{E}}_t\left[\exp \left\{12 p L_2 
            \hat{\sigma}^2 \zeta^{(m, q)}\right\}\right] .
        \end{align*}
    \end{small}

    Using Theorem \ref{Doob-inequality}, Remark \ref{Doob-exp} and H\"{o}lder's inequality under $G$-framework, we obtain 
    that for each $p\geq1$ and $t\in[0,T],$
    \begin{small}
        \begin{align}
            &\widehat{\mathbb{E}}\left[\exp\left\{3pL_2\tilde{\sigma}^2
            \sup_{t \in[t_0, t_0+h]}\delta_\theta \bar{y}_t^{(m, q)}
            \right\}\right]\notag\\
            \leq& \widehat{\mathbb{E}}\Biggl[\sup_{t \in[t_0, t_0+h]}\widehat{\mathbb{E}}_t 
            \Biggl[\exp\Biggl\{6 p L_2 \hat{\sigma}^2
            \left(|\xi|+\rho(\theta, m)+\chi^{(m, q)}+L_1h\left(\sup_{t \in[t_0, t_0+h]} 
            \delta_\theta \bar{Y}_t^{(m, q)}\right)\right)\Biggr\}\Biggr]^2\notag\\
            &\qquad\times\sup_{t \in[t_0, t_0+h]}\widehat{\mathbb{E}}_t\left[\exp \left\{24 p L_2 
            \hat{\sigma}^2 \zeta^{(m, q)}\right\}\right]\Biggr]\notag\\
            \leq&\hat{A}_G\Biggl(\widehat{\mathbb{E}}\Biggl[\exp\Biggl\{
                48p L_2 \hat{\sigma}^2\Biggl(
                    |\xi|+\rho(\theta, m)+\chi^{(m, q)}+L_1h\left(\sup_{t \in[t_0, t_0+h]} 
            \delta_\theta \bar{Y}_t^{(m, q)}\right)
                \Biggr)
            \Biggr\}\Biggr]\Biggr)^{\frac{1}{2}}\notag\\
            &\qquad\times\left(\widehat{\mathbb{E}}\left[\exp \left\{96 p L_2 
            \hat{\sigma}^2 \zeta^{(m, q)}\right\}\right]\right)^{\frac{1}{2}}.\label{bar-y-mq}
        \end{align}
    \end{small}

    Set 
    \begin{small}
        \begin{align*}
            C_8(\xi)=\sup_{t \in[t_0, t_0+h]}\left|\bar{A}_t^3\right|+\frac{6C}{c}\sup_{t \in[t_0, t_0+h]}
            \widehat{\mathbb{E}}\left[\left|y_t^{(m, q)}\right|+\left|y_t^{(m)}\right|\right]
            +\frac{6C}{c}\widehat{\mathbb{E}}\left[|\xi|\right].
        \end{align*}
    \end{small}By similar idea as the proof of inequality \eqref{delta-theta-Y}, we have 
    \begin{small}
        \begin{align*}
            &\left|\delta_{\theta}Y_{t}^{(m, q)}\right|\leq C_8(\xi)+\left|\delta_{\theta}y_{t}^{(m, q)}\right|
            +\frac{3C}{c}\sup_{t \in[t_0, t_0+h]}\widehat{\mathbb{E}}\left[\left|\delta_{\theta}y_{t}^{(m, q)}\right|\right],\\
            &\left|\delta_{\theta}\tilde{Y}_{t}^{(m, q)}\right|\leq C_8(\xi)+\left|\delta_{\theta}\tilde{y}_{t}^{(m, q)}\right|
            +\frac{3C}{c}\sup_{t \in[t_0, t_0+h]}\widehat{\mathbb{E}}\left[\left|\delta_{\theta}\tilde{y}_{t}^{(m, q)}\right|\right],
        \end{align*}
    \end{small}which yields that 
    \begin{small}
        \begin{align*}
            \delta_{\theta}\bar{Y}_{t}^{(m, q)}\leq 2C_8(\xi)+\delta_{\theta}\bar{y}_{t}^{(m, q)}
            +\frac{6C}{c}\sup_{t \in[t_0, t_0+h]}\widehat{\mathbb{E}}\left[\delta_{\theta}\bar{y}_{t}^{(m, q)}\right].
        \end{align*}
    \end{small}Then, we have
    \begin{small}
        \begin{align*}
            &\widehat{\mathbb{E}}\left[\exp\left\{3pL_2\tilde{\sigma}^2
            \sup_{t \in[t_0, t_0+h]}\delta_{\theta}\bar{Y}_{t}^{(m, q)}\right\}\right]\\
            \leq&\exp\left\{6pL_2\tilde{\sigma}^2C_8(\xi)\right\}\widehat{\mathbb{E}}
            \left[\exp\left\{3pL_2\tilde{\sigma}^2\sup_{t \in[t_0, t_0+h]}\delta_{\theta}\bar{y}_{t}^{(m, q)}\right\}\right]
            \widehat{\mathbb{E}}\left[\exp\left\{18\frac{C}{c}pL_2\tilde{\sigma}^2\sup_{t \in[t_0, t_0+h]}
            \delta_{\theta}\bar{y}_{t}^{(m, q)}\right\}\right]\\
            \leq&\exp\left\{6pL_2\tilde{\sigma}^2C_8(\xi)\right\}\widehat{\mathbb{E}}\left[\exp
            \left\{\left(6+36\frac{C}{c}\right)pL_2\tilde{\sigma}^2\sup_{t \in[t_0, t_0+h]}
            \delta_{\theta}\bar{y}_{t}^{(m, q)}\right\}\right]\\
            \leq&\hat{A}_G\Biggl(\widehat{\mathbb{E}}\Biggl[\exp\Biggl\{
                \left(96+576\frac{C}{c}\right)p L_2 
                \hat{\sigma}^2\Biggl(|\xi|+\rho(\theta, m)+C_8(\xi)+\chi^{(m, q)}+L_1h\left(\sup_{t \in[t_0, t_0+h]} 
                \delta_\theta \bar{Y}_t^{(m, q)}\right)\Biggr)
            \Biggr\}\Biggr]\Biggr)^{\frac{1}{2}}\\
            &\qquad\times \left(\widehat{\mathbb{E}}\left[\exp \left\{\left(192+1152\frac{C}{c}\right) p L_2 
            \hat{\sigma}^2 \zeta^{(m, q)}\right\}\right]\right)^{\frac{1}{2}}\\
            \leq&\hat{A}_G\left(\widehat{\mathbb{E}}\left[\exp\left\{
                \left(192+1152\frac{C}{c}\right)pL_2\hat{\sigma}^2\rho(\theta, m)
            \right\}\right]\right)^{\frac{1}{4}}\times 
            \left(\widehat{\mathbb{E}}\left[\exp \left\{\left(192+1152\frac{C}{c}\right) p L_2 
            \hat{\sigma}^2 \zeta^{(m, q)}\right\}\right]\right)^{\frac{1}{2}}\\
            &\qquad\times \left(\widehat{\mathbb{E}}\left[\exp\left\{
                \left(384+2304\frac{C}{c}\right)pL_2\hat{\sigma}^2
            \left(L_1h\left(\sup_{t \in[t_0, t_0+h]} 
            \delta_\theta \bar{Y}_t^{(m, q)}\right)\right)\right\}\right]\right)^{\frac{1}{8}}\\
            &\qquad\times\left(\widehat{\mathbb{E}}\left[\exp\left\{
                \left(384+2304\frac{C}{c}pL_2\hat{\sigma}^2\left(|\xi|+C_8(\xi)+\chi^{(m, q)}\right)\right)
            \right\}\right]\right)^{\frac{1}{8}}.
        \end{align*}
    \end{small}Let $h$ such that $(128+768C/c)L_1h\hat{\sigma}^2/\tilde{\sigma}^2<1,$ 
    then for the small time interval $[t_0,t_0+h]$, we have
    \begin{small}
        \begin{align*}
            &\widehat{\mathbb{E}}\left[\exp\left\{3pL_2\tilde{\sigma}^2
            \sup_{t \in[t_0, t_0+h]}\delta_{\theta}\bar{Y}_{t}^{(m, q)}\right\}\right]\\
            \leq& \hat{A}_p\left(\widehat{\mathbb{E}}\left[\exp\left\{
                \left(192+1152\frac{C}{c}\right)pL_2\hat{\sigma}^2\rho(\theta, m)
            \right\}\right]\right)^{\frac{1}{4}}\\
            &\qquad\times\left(\widehat{\mathbb{E}}\left[\exp\left\{3pL_2\tilde{\sigma}^2
            \sup_{t \in[t_0, t_0+h]}\delta_{\theta}\bar{Y}_{t}^{(m, q)}\right\}\right]\right)^{\left(
                16+96\frac{C}{c}L_1h\hat{\sigma}^2/\tilde{\sigma}^2\right)},
        \end{align*}
    \end{small}where 
    \begin{small}
        \begin{align*}
            \hat{A}_p&=\hat{A}_G\sup_{m,q\leq1}\Biggl[\left(\widehat{\mathbb{E}}\left[\exp\left\{
                \left(384+2304\frac{C}{c}pL_2\hat{\sigma}^2\left(|\xi|+C_8+\chi^{(m, q)}\right)\right)
            \right\}\right]\right)^{\frac{1}{8}}\\
            &\qquad\times \left(\widehat{\mathbb{E}}\left[\exp \left\{\left(192+1152\frac{C}{c}\right) p L_2 
            \hat{\sigma}^2 \zeta^{(m, q)}\right\}\right]\right)^{\frac{1}{2}}\Biggr]<\infty.
        \end{align*}
    \end{small}
    
    It follows that 
    \begin{small}
        \begin{align*}
            &\widehat{\mathbb{E}}\left[\exp\left\{3pL_2\tilde{\sigma}^2
            \sup_{t \in[t_0, t_0+h]}\delta_{\theta}\bar{Y}_{t}^{(m, q)}\right\}\right]\\
            \leq&\hat{A}_p^{\frac{1}{1-(128+768C/c)L_1h\hat{\sigma}^2/\tilde{\sigma}^2}}
            \left(\widehat{\mathbb{E}}\left[\exp\left\{
                    \left(192+1152\frac{C}{c}\right)pL_2\hat{\sigma}^2\rho(\theta, m)
                \right\}\right]\right)^{\frac{1}{4\left(1-(128+768C/c)L_1h\hat{\sigma}^2/\tilde{\sigma}^2\right)}}\\
            \leq&\hat{A}_p^{\frac{1}{1-(128+768C/c)L_1h\hat{\sigma}^2/\tilde{\sigma}^2}}
            \Biggl(\widehat{\mathbb{E}}\Biggl[\exp\Biggl\{
                \frac{\left(192+1152\frac{C}{c}\right)pL_2\hat{\sigma}^2}{1-\theta}\\
                &\qquad\times \left((|\xi|-m)^{+}+2\int_{t}^{t_0+h}\left(\left|g_0(s)\right|-m\right)^{+}ds\right)
            \Biggr\}\Biggr]\Biggr)^{\frac{1}{4\left(1-(128+768C/c)L_1h\hat{\sigma}^2/\tilde{\sigma}^2\right)}}.
        \end{align*}
    \end{small}
    
    {\bf Step 4. The convergence on small time interval $\mathbf{[t_0,t_0+h]}$.} It is easy to verify 
    that as $m\to\infty$, we have 
    \begin{small}
        \begin{align*}
            \exp\Biggl\{
                    \frac{\left(192+1152\frac{C}{c}\right)pL_2\hat{\sigma}^2}{1-\theta}
                    \left((|\xi|-m)^{+}+2\int_{t}^{t_0+h}\left(\left|g_0(s)\right|-m\right)^{+}ds\right)
                \Biggr\}\in L_G^1(\Omega)\downarrow 1.
        \end{align*}
    \end{small}By the Assertion $(ii)$ of Theorem \ref{monotone-con-thm}, for each $p\geq1$ and $\theta\in(0,1)$, 
    we get 
    \begin{small}
        \begin{align*}
            \lim_{m\to\infty}\widehat{\mathbb{E}}\left[\exp\Biggl\{
                    \frac{\left(192+1152\frac{C}{c}\right)pL_2\hat{\sigma}^2}{1-\theta}
                    \left((|\xi|-m)^{+}+2\int_{t}^{t_0+h}\left(\left|g_0(s)\right|-m\right)^{+}ds\right)
                \Biggr\}\right]= 1.
        \end{align*}
    \end{small}Hence, we obtain that 
    \begin{small}
        \begin{align}
            \limsup_{m\to\infty}\sup_{q\geq1}\widehat{\mathbb{E}}\left[\exp\left\{3pL_2\tilde{\sigma}^2
            \sup_{t \in[t_0, t_0+h]}\delta_{\theta}\bar{Y}_{t}^{(m, q)}\right\}\right]\leq
            \hat{A}_p^{\frac{1}{1-(128+768C/c)L_1h\hat{\sigma}^2/\tilde{\sigma}^2}}.\label{bar-Y-mq}
        \end{align}
    \end{small}
    
    Note the fact that 
    \begin{small}
        \begin{align*}
            Y^{(m, q)}-Y^{(m)}=(1-\theta)(\delta_{\theta}Y^{(m, q)}-Y^{(m)}),
        \end{align*}
    \end{small}Thus, we have 
    \begin{small}
        \begin{align*}
            &\limsup_{m\to\infty}\sup_{q\geq1}\widehat{\mathbb{E}}\left[\sup_{t\in[t_0, t_0+h]}
            \left|Y_t^{(m, q)}-Y_t^{(m)}\right|^n\right]\\
            =&\limsup_{m\to\infty}\sup_{q\geq1}\widehat{\mathbb{E}}\left[(1-\theta)^n
            \sup_{t\in[t_0, t_0+h]}\left|\delta_{\theta}Y_t^{(m, q)}-Y_t^{(m)}\right|^n\right]\\
            \leq&\limsup_{m\to\infty}\sup_{q\geq1}2^{n-1}(1-\theta)^n\widehat{\mathbb{E}}
            \left[\left(\sup_{t\in[t_0, t_0+h]}\delta_{\theta}\bar{Y}_{t}^{(m, q)}\right)^n+
            \sup_{t\in[t_0, t_0+h]}\left|Y_t^{(m)}\right|^n\right]\\
            \leq&\limsup_{m\to\infty}\sup_{q\geq1}2^{n-1}(1-\theta)^n\left(
                \frac{n!}{L_2^n}\widehat{\mathbb{E}}\left[\exp\left\{L_2
                \sup_{t\in[t_0, t_0+h]}\delta_{\theta}\bar{Y}_{t}^{(m, q)}\right\}\right]
                +\sup_{m\geq1}\widehat{\mathbb{E}}\left[\sup_{t\in[t_0, t_0+h]}\left|Y_t^{(m)}\right|^n\right]
            \right)\\
            \leq&2^{n-1}(1-\theta)^n\left(\frac{n!}{3\tilde{\sigma}^2L_2^n}
            \hat{A}_1^{\frac{1}{1-(128+768C/c)L_1h\hat{\sigma}^2/\tilde{\sigma}^2}}+
            \sup_{m\geq1}\widehat{\mathbb{E}}\left[\sup_{t\in[t_0, t_0+h]}\left|Y_t^{(m)}\right|^n\right]\right).
        \end{align*}
    \end{small}By the arbitrariness of $\theta$ and \eqref{estimate-Y-m}, we could 
    find a continuous process $Y$ such that 
    \begin{small}
        \begin{align}\label{Y-m-converge-Y}
            \lim_{m\to\infty}\widehat{\mathbb{E}}\left[\sup_{t\in[t_0, t_0+h]}
                \left|Y_t^{(m)}-Y_t\right|^n\right]=0, \forall n\geq1.
        \end{align}
    \end{small}

    Applying It\^{o}'s formula to $|Y_t^{(m,q)}-Y_t^{(m)}|^2$, we have
    \begin{small}
        \begin{align}\label{Zmq-Zm}
            &\widehat{\mathbb{E}}\left[\int_{t_0}^{t_0+h}\left|Z_t^{(m,q)}-Z_t^{(m)}
            \right|^2dt\right]\notag\\
            &\leq C_9 \widehat{\mathbb{E}}\left[\sup_{t\in[t_0, t_0+h]}
            \left|Y_t^{(m, q)}-Y_t^{(m)}\right|^2+\sup_{t\in[t_0, t_0+h]}
            \left|Y_t^{(m, q)}-Y_t^{(m)}\right|\Gamma^{(m,q)}\right]\\
            &\leq C_9 \widehat{\mathbb{E}}\left[\sup_{t\in[t_0, t_0+h]}
            \left|Y_t^{(m, q)}-Y_t^{(m)}\right|^2\right]+C_9\left(\widehat{\mathbb{E}}\left[
                \left|\Gamma^{(m,q)}\right|^2
            \right]\right)^{\frac{1}{2}}\left(\widehat{\mathbb{E}}\left[\sup_{t\in[t_0, t_0+h]}
            \left|Y_t^{(m, q)}-Y_t^{(m)}\right|^2\right]\right)^{\frac{1}{2}},\notag
        \end{align}
    \end{small}where 
    \begin{small}
        \begin{align*}
            \Gamma^{(m,q)}=&\int_{t_0}^{t_0+h}\left|g^{(m, q)}\left(t,Y_t^{(m, q)},Z_t^{(m,q)}\right)
            -g^{(m, q)}\left(t,Y_t^{(m)},Z_t^{(m)}\right)\right|d \langle B\rangle_t\\
            &+\left|K_{t_0+h}^{(m,q)}\right|+\left|K_{t_0+h}^{(m)}\right|+\left|A_{t_0+h}^{(m,q)}\right|
            +\left|A_{t_0+h}^{(m)}\right|.
        \end{align*}
    \end{small}In spirit of Assumption \ref{assumption-UBT} and Inequality \eqref{estimate-Y-m},
    we see that $\sup\nolimits_{m,q}\widehat{\mathbb{E}}[|\Gamma^{(m,q)}|^2]<\infty.$
    Thus, by \eqref{Y-m-converge-Y} and \eqref{Zmq-Zm}, there is a process $Z\in M_G^2(t_0,t_0+h)$
    such that 
    \begin{small}
        \begin{align*}
            \lim_{m\to\infty}\widehat{\mathbb{E}}\left[
                \int_{t_0}^{t_0+h}\left|Z_t^{(m)}-Z_t
                \right|^2dt
            \right]=0.
        \end{align*}
    \end{small}By Assertion $(i)$ of Theorem \ref{monotone-con-thm} and \eqref{estimate-Y-m},
    we have that $Z\in \mathcal{H}_G(\mathbb{R})$, i.e., for any $n\geq1$
    \begin{small}
        \begin{align*}
            \widehat{\mathbb{E}}\left[
                \left(\int_{t_0}^{t_0+h}\left|Z_t
                \right|^2dt\right)^n
            \right]<\infty,
        \end{align*}
    \end{small}which together with Lemma \ref{lemma-HTW22} implies that 
    \begin{small}
        \begin{align*}
            \lim_{m\to\infty}\widehat{\mathbb{E}}\left[
                \left(\int_{t_0}^{t_0+h}\left|Z_t^{(m)}-Z_t
                \right|^2dt\right)^n
            \right]<\infty,\ \forall n\geq1.
        \end{align*}
    \end{small}

    By the similar idea from the proof of Theorem \ref{thm-BT}, we have 
    \begin{small}
        \begin{equation*}%\label{AA313'}
            \begin{split}
            &\left|(A_{t_0+h}^{(m,q)}-A_{t}^{(m,q)})-(A_{t_0+h}^{(m)}-A_{t}^{(m)})\right|
            =\left|\bar{A}^{(m,q)}_{t_0+h-t}-\bar{A}^{(m)}_{t_0+h-t}\right|\\
            \leq &\frac{C}{c}\sup_{t\in[t_0,t_0+h]}|\bar{s}^{(m,q)}_t-\bar{s}^{(m)}_t|
            +\frac{1}{c}\sup_{(t,x)\in[t_0,t_0+h]\times \mathbb{R}}\left(
            |\widehat{\bar{l}}^{(m,q)}(t,x)|\vee|\widehat{\bar{r}}^{(m,q)}(t,x)|\right)\\
            =&\frac{C}{c}\sup_{t\in[t_0,t_0+h]}\Big|{s}^{(m,q)}_{t_0+h}-{s}^{(m,q)}_{t_0+h-t}-
            \big({s}^{(m)}_{t_0+h}-{s}^{(m)}_{t_0+h-t}\big)\Big|+
            \frac{1}{c}\sup_{(t,x)\in[t_0,t_0+h]\times \mathbb{R}}
            \Big(|\widehat{l}^{(m,q)}(t,x)|\vee|\widehat{r}^{(m,q)}(t,x)|\Big)\\
            \leq&C_{10}\sup_{t\in[t_0,t_0+h]}\widehat{\mathbb{E}}\left[
                \left|y_t^{(m,q)}-y_t^{(m)}\right|\right],
            \end{split}
        \end{equation*}
    \end{small}where $\hat{h}^{(m,q)}=h^{(m,q)}(t,x)-h^{(m)}(t,x)$, for $h=l,r,\bar{l},\bar{r}$.
    Moreover, we derive that for any $n\geq1$
    \begin{small}
        \begin{align*}
            \sup_{t\in[t_0,t_0+h]}\left|A_{t}^{(m,q)}-A_{t}^{(m)}\right|&\leq
            \sup_{t\in[t_0,t_0+h]}\left|(A_{t_0+h}^{(m,q)}-A_{t}^{(m,q)})-(A_{t_0+h}^{(m)}-A_{t}^{(m)})\right|+\left|A_{t_0+h}^{(m,q)}-A_{t_0+h}^{(m)}\right|\\
            &\leq C_{11} \sup_{t\in[t_0,t_0+h]}\widehat{\mathbb{E}}\left[
                \left|y_t^{(m,q)}-y_t^{(m)}\right|\right].
        \end{align*}
    \end{small}Note that $y^{(m)}$ can be seen as a solution to a standard quadratic $G$-BSDE \eqref{y-m} with terminal $\xi$ and 
    generator $g^{(m)}(t,Y_t^{(m)},z_t^{(m)})$. 
    %Similar to show $Y^{(m)}$ is a Cauchy sequence, we can prove 
    We still need to show that $y^{(m)}$ is a Cauchy sequence. Due to $p\geq1$ is arbitrary, $\tilde{\sigma}^2$ 
    can be replaced by $\hat{\sigma}^2$ here.
    By \eqref{bar-y-mq}, we have
    \begin{small}
        \begin{align*}
            &\widehat{\mathbb{E}}\left[\exp\left\{3pL_2\tilde{\sigma}^2
            \sup_{t \in[t_0, t_0+h]}\delta_\theta \bar{y}_t^{(m, q)}
            \right\}\right]\\
            \leq&\hat{A}_G\Biggl(\widehat{\mathbb{E}}\Biggl[\exp\Biggl\{
                48p L_2 \hat{\sigma}^2\Biggl(
                    |\xi|+\rho(\theta, m)+\chi^{(m, q)}+L_1h\left(\sup_{t \in[t_0, t_0+h]} 
            \delta_\theta \bar{Y}_t^{(m, q)}\right)
                \Biggr)
            \Biggr\}\Biggr]\Biggr)^{\frac{1}{2}}\\
            &\qquad\times\left(\widehat{\mathbb{E}}\left[\exp \left\{96 p L_2 
            \hat{\sigma}^2 \zeta^{(m, q)}\right\}\right]\right)^{\frac{1}{2}}\\
            \leq&\hat{A}_G\left(\widehat{\mathbb{E}}\left[\exp\left\{96p L_2 \hat{\sigma}^2
            \left(|\xi|+\rho(\theta, m)+\chi^{(m, q)}\right)\right\}\right]\right)^{\frac{1}{2}}\\
            &\qquad\times\left(\widehat{\mathbb{E}}\left[\exp\left\{96p L_2 \hat{\sigma}^2
            (L_1h+1)\left(\sup_{t \in[t_0, t_0+h]} 
            \delta_\theta \bar{Y}_t^{(m, q)}\right)\right\}\right]\right)^{\frac{1}{2}}\\
            &\qquad\times\left(\widehat{\mathbb{E}}\left[\exp \left\{96 p L_2 
            \hat{\sigma}^2 \zeta^{(m, q)}\right\}\right]\right)^{\frac{1}{2}}\\
            \leq&\hat{A}_G\left(\widehat{\mathbb{E}}\left[\exp\left\{192p L_2 \hat{\sigma}^2
            \rho(\theta, m)\right\}\right]\right)^{\frac{1}{4}}\left(\widehat{\mathbb{E}}
            \left[\exp\left\{192p L_2 \hat{\sigma}^2
            \left(|\xi|+\chi^{(m, q)}\right)\right\}\right]\right)^{\frac{1}{4}}\\
            &\qquad\times\left(\widehat{\mathbb{E}}\left[\exp\left\{96p L_2 \hat{\sigma}^2
            (L_1h+1)\left(\sup_{t \in[t_0, t_0+h]} 
            \delta_\theta \bar{Y}_t^{(m, q)}\right)\right\}\right]\right)^{\frac{1}{2}}\\
            &\qquad\times\left(\widehat{\mathbb{E}}\left[\exp \left\{96 p L_2 
            \hat{\sigma}^2 \zeta^{(m, q)}\right\}\right]\right)^{\frac{1}{2}}.
        \end{align*}
    \end{small}Therefore, by \eqref{bar-Y-mq}, we get
    \begin{small}
        \begin{align}
            &\lim_{m\to\infty}\widehat{\mathbb{E}}\left[\exp\left\{3pL_2\tilde{\sigma}^2
            \sup_{t \in[t_0, t_0+h]}\delta_\theta \bar{y}_t^{(m, q)}
            \right\}\right] \notag\\
            \leq&\hat{A}_G \sup_m\left(\widehat{\mathbb{E}}
            \left[\exp\left\{192p L_2 \hat{\sigma}^2
            \left(|\xi|+\chi^{(m, q)}\right)\right\}\right]\right)^{\frac{1}{4}}\notag\\
            &\qquad\times \limsup_{m\to\infty}\left(\widehat{\mathbb{E}}\left[\exp\left\{96p L_2 \hat{\sigma}^2
            (L_1h+1)\left(\sup_{t \in[t_0, t_0+h]} 
            \delta_\theta \bar{Y}_t^{(m, q)}\right)\right\}\right]\right)^{\frac{1}{2}}\notag\\
            &\qquad\times\sup_m\left(\widehat{\mathbb{E}}\left[\exp \left\{96 p L_2 
            \hat{\sigma}^2 \zeta^{(m, q)}\right\}\right]\right)^{\frac{1}{2}}\notag\\
            \leq&\hat{A}_G C_{12}.\label{limit-bar-y-mq}
        \end{align}
    \end{small}

    Note that $y^{(m, q)}-y^{(m)}=(1-\theta)(\delta_{\theta}y^{(m, q)}-y^{(m)})$.  It follows that
    \begin{small}
        \begin{align*}
            &3\tilde{\sigma}^2\widehat{\mathbb{E}}\left[\sup_{t\in[t_0,t_0+h]}
                \left|y_t^{(m,q)}-y_t^{(m)}\right|\right]\\
                \leq&(1-\theta)\left(\frac{1}{L_2}\sup_{\theta\in(0,1)}\widehat{\mathbb{E}}\left[\exp\left\{3L_2\tilde{\sigma}^2
                \sup_{t \in[t_0, t_0+h]}\delta_\theta \bar{y}_t^{(m, q)}
                \right\}\right]+3\tilde{\sigma}^2\widehat{\mathbb{E}}\left[\sup_{t \in[t_0, t_0+h]}
                \left|y_t^{(m)}\right|\right]\right).  
        \end{align*}
    \end{small}Using \eqref{limit-bar-y-mq} and \eqref{estimate-y-m},
    let $m\to\infty$ and $\theta\to1$ in order, we conclude that 
    \begin{small}
        \begin{align*}
            \lim_{m\to\infty}\widehat{\mathbb{E}}\left[\sup_{t\in[t_0,t_0+h]}
            \left|y_t^{(m,q)}-y_t^{(m)}\right|\right]=0.
        \end{align*}
    \end{small}Hence, we have
        \begin{small}
            \begin{align*}
                \lim_{m\to\infty}\widehat{\mathbb{E}}\left[\sup_{t\in[t_0,t_0+h]}\left|A_{t}^{(m,q)}-A_{t}^{(m)}\right|\right]=0.
            \end{align*}
        \end{small}

    Consequently, by the convergences of $Y^{(m)}$ and $Z^{(m)}$, there are three processes 
    $(Y,Z,A)\in \mathcal{E}_G(\mathbb{R})\times \mathcal{H}_G(\mathbb{R})\times \mathcal{A}_D$ on $[t_0,t_0+h]$
    such that for any $n\geq1$
    \begin{small}
        \begin{align}\label{Y-Z-A-0}
            \lim_{m\to\infty}\widehat{\mathbb{E}}\left[\sup_{t\in[t_0,t_0+h]}
            \left|Y_t^{(m)}-Y_t\right|^n+\left(\int_{t_0}^{t_0+h}\left|Z_t^{(m)}-Z_t
            \right|^2dt\right)^n+\sup_{t\in[t_0,t_0+h]}\left|A_{t}^{(m)}-A_t\right|\right]
            =0,\ \forall n\geq1.
        \end{align}
    \end{small}

    For $t\in[t_0,t_0+h]$, set 
    \begin{small}
        \begin{align*}
            K_t=Y_t-Y_0+\int_{t_0}^{t}g(s,Y_s,Z_s)d\langle B\rangle_s-\int_{t_0}^{t}Z_sdB_s+A_t.
        \end{align*}
    \end{small}By Assumption \ref{assumption-UBT} and \eqref{Y-Z-A-0}, using the same method in the proof
    of Theorem 3.9 in \cite{HTW22}, we get 
    \begin{small}
        \begin{align*}%\label{gm-g}
            \lim_{m\to\infty}\widehat{\mathbb{E}}\left[\left(\int_{t_0}^{t_0+h}\left|g^{(m)}(t,Y_t^{(m)},Z_t^{(m)})-g(t,Y_t,Z_t)
            \right|\langle B\rangle_t\right)^n\right]=0.
        \end{align*}
    \end{small}Therefore, 
    \begin{small}
        \begin{align*}
            \lim_{m\to\infty}\widehat{\mathbb{E}}\left[\left|K_t^{(m)}-K_t\right|^n\right]=0,\ \forall n\geq1.
        \end{align*}
    \end{small}Thus $K$ is a non-increasing $G$-martingale, and then 
    $(Y,Z,K,A)\in \mathcal{E}_G(\mathbb{R})\times \mathcal{H}_G(\mathbb{R})\times \mathcal{L}_G(\mathbb{R})\times \mathcal{A}_D$ satisfies
    \eqref{DMFGBSDE-simplify} on the small time interval $[t_0,t_0+h]$.

    The proof is complete.   
\end{proof}

\renewcommand\thesection{\normalsize Appendix: Technical proofs}
\section{ }

\renewcommand\thesection{A}
\begin{proof}[Proof of equations \eqref{estimate-y-m} and \eqref{estimate-Y-m}]
Denote
\begin{small}
    \begin{align*}
        C_2(\xi)=|\xi| +\int_{t}^{t_0+h}\alpha_tdt+\frac{L_2}{2}h.
    \end{align*}
\end{small}Then, by Lemma 
    \ref{lemma-priori-estimate}, for any $p\geq 1$, we have 
    \begin{small}
        \begin{align}\label{uniform-y-m}
            \exp\left\{3pL_2\tilde{\sigma}^2\left|y_t^{(m)}\right|\right\}
            &\leq \widehat{\mathbb{E}}_t\left[\exp\left\{3pL_2\hat{\sigma}^2|\xi|
            +3pL_2\hat{\sigma}^2\int_{t}^{t_0+h}\alpha_s+\frac{L_2}{2}+L_1\left|Y_s^{(m)}\right|ds\right\}\right]\notag\\
            &\leq \widehat{\mathbb{E}}_t\left[\exp\left\{3pL_2\hat{\sigma}^2
            \left(C_2(\xi)+L_1h\sup_{s\in[t,t_0+h]}\left|Y_s^{(m)}\right|\right)\right\}\right]
            %&\leq \widehat{\mathbb{E}}_t\left[\exp\left\{3pL_2\hat{\sigma}^2
            %\left(C_2+L_1(T-t)\sup_{s\in[t,T]}\left|y_s^{(m)}+\bar{A}_{T-s}^{(m)}\right|\right)\right\}\right]
        \end{align}
    \end{small}By the representation of $Y^{(m)}$, i.e. equation \eqref{representation-Y-m}, we get 
    \begin{small}
        \begin{align*}
            \sup_{s\in[t,t_0+h]}\left|Y_s^{(m)}\right|
            &=\sup_{s\in[t,t_0+h]}\left|y_s^{(m)}+\bar{A}_{T-s}^{(m)}\right|\\
            &\leq \sup_{s\in[t,t_0+h]}\left|y_s^{(m)}\right|
            +\sup_{s\in[t,t_0+h]}\left|\bar{A}_{t_0+h-s}^{(m)}-\bar{A}_{t_0+h-s}^3\right|
            +\sup_{s\in[t,t_0+h]}\left|\bar{A}_{t_0+h-s}^3\right|,
        \end{align*}
    \end{small}where $\bar{A}^3$ is the second component of the solution to 
    the Skorokhod problem $\mathbb{SP}_{\bar{l}^3}^{\bar{r}^3}(\bar{s}^3)$ 
    and $\bar{s}^3=\widehat{\mathbb{E}}[\xi]$,  
    \begin{small}
        \begin{align*}
            &\bar{l}^3(t,x)=l^3(t_0+h-t,x)=\widehat{\mathbb{E}}\left[L(t_0+h-t,x)\right], \\
            &\bar{r}^3(t,x)=r^3(t_0+h-t,x)=\widehat{\mathbb{E}}\left[R(t_0+h-t,x)\right].
            \end{align*}
    \end{small}By Theorem \ref{thm-SP-K1-K2}, we have
    \begin{small}
        \begin{align*}
            \sup_{t\in[t,t_0+h]}\left|\bar{A}_{t_0+h-t}^{(m)}-\bar{A}_{t_0+h-t}^3\right|
            \leq \frac{C}{c}\sup_{t\in[0,h]} \left|\bar{s}_{t}^{(m)}-\bar{s}_{t}^3\right|
            +\frac{1}{c}\left(\bar{L}^3_T\vee \bar{R}^3_T\right),
        \end{align*}
    \end{small}where $\bar{L}^3_h=\sup\nolimits_{(t,x)\in[0,h]\times\mathbb{R}}\left|\bar{l}^{(m)}(t,x)-\bar{l}^{3}(t,x)\right|$ and
    $\bar{R}^3_h=\sup\nolimits_{(t,x)\in[0,h]\times\mathbb{R}}\left|\bar{r}^{(m)}(t,x)-\bar{r}^{3}(t,x)\right|$.
    
    By \eqref{bar-slr-m}, we get
    \begin{small}
        \begin{align*}
            \sup_{t\in[0,h]} \left|\bar{s}_{t}^{(m)}-\bar{s}_{t}^3\right|\leq 
            \sup_{t\in[0,h]}\left|\widehat{\mathbb{E}}[y^{(m)}_{t_0+h-t}]-\widehat{\mathbb{E}}[\xi]\right|
            \leq \sup_{t\in[t_0,t_0+h]}\left|\widehat{\mathbb{E}}[y^{(m)}_{t}]\right|+\widehat{\mathbb{E}}[\xi],
        \end{align*}
    \end{small}and 
    \begin{small}
        \begin{align*}
            \bar{L}_T^3&=\sup_{(t,x)\in[0,h]\times\mathbb{R}}
            \left|\bar{l}^{(m)}(t,x)-\bar{l}^{3}(t,x)\right|\\
              &=\sup_{(t,x)\in[0,h]\times\mathbb{R}} 
              \left|\widehat{\mathbb{E}}\left[L(t_0+h-t,y^{(m)}_{t_0+h-t}-
              \widehat{\mathbb{E}}[y^{(m)}_{t_0+h-t}]+x)\right]
              -\widehat{\mathbb{E}}\left[L(t_0+h-t,x)\right]\right|\\
             &\leq 2C\sup_{t\in [0,h]}\widehat{\mathbb{E}}[|y_{t_0+h-t}^{(m)}|]=
             2C\sup_{t\in [t_0,t_0+h]}\widehat{\mathbb{E}}[|y^{(m)}_{t}|].
        \end{align*}
    \end{small}Similar analysis holds for $\bar{R}_T^3$. Then, we derive that 
    \begin{small}
        \begin{equation}\label{estimate-bar-A}
            \sup_{t\in[t_0,t_0+h]}\left|\bar{A}^{(m)}_{t_0+h-t}-\bar{A}^{3}_{t_0+h-t}\right|
            \leq 3\frac{C}{c}\sup_{t\in [t_0,t_0+h]}\widehat{\mathbb{E}}[|y^{(m)}_{t}|]
            +\frac{C}{c}\widehat{\mathbb{E}}\left[\left|\xi\right|\right].    
        \end{equation}
    \end{small} 
    
    Hence, we obtain that 
    \begin{small}
        \begin{equation}\label{sup-Y-m}
            \sup_{t\in[t_0,t_0+h]}\left|Y_t^{(m)}\right|
            \leq\sup_{t\in[t_0, t_0+h]}|y^{(m)}_t|+
            3\frac{C}{c}\sup_{t\in[t_0, t_0+h]}\widehat{\mathbb{E}}[|y^{(m)}_t|]+C_3(\xi),
        \end{equation} 
    \end{small}where $C_3(\xi)=\frac{C}{c}\widehat{\mathbb{E}}\left[\left|\xi\right|\right]+\sup_{t\in[0,T]}\left|\bar{A}_t^3\right|.$ 
    Substituting \eqref{sup-Y-m} into \eqref{uniform-y-m},
    we have 
    \begin{small}
        \begin{align*}
            \exp\left\{3pL_2\tilde{\sigma}^2\left|y_t^{(m)}\right|\right\}&\leq
            \widehat{\mathbb{E}}_t\Biggl[
                \exp\Biggl\{3pL_2\hat{\sigma}^2\Biggl(C_2(\xi,p)+L_1hC_3(\xi,p)\\
                &\qquad+L_1h\sup_{t\in[t_0,t_0+h]}\biggl(|y^{(m)}_t|+
                3\frac{C}{c}\widehat{\mathbb{E}}[|y^{(m)}_t|]\biggr)\Biggr)\Biggr\}\Biggr]\\
                &\leq \widehat{\mathbb{E}}_t\left[C_4(\xi,p)\exp\left\{
                    3pL_2\hat{\sigma}^2L_1h\sup_{t\in[t_0,t_0+h]}|y^{(m)}_t|\right\}\right]\\
                &\qquad\times \widehat{\mathbb{E}}\left[\exp\left\{9pL_2\hat{\sigma}^2L_1h
                \frac{C}{c}\sup_{t\in[t_0,t_0+h]}|y^{(m)}_t|\right\}\right]\\
                &\leq \left(\widehat{\mathbb{E}}_t\left[C_4(\xi,p)\right]\right)^{\frac{1}{2}}
                \left(\widehat{\mathbb{E}}_t\left[\exp\left\{6pL_2\hat{\sigma}^2L_1h
                \sup_{t\in[t_0,t_0+h]}|y^{(m)}_t|\right\}\right]\right)^{\frac{1}{2}}\\
                &\qquad\times \widehat{\mathbb{E}}\left[\exp\left\{9pL_2\hat{\sigma}^2L_1h
                \frac{C}{c}\sup_{t\in[t_0,t_0+h]}|y^{(m)}_t|\right\}\right],
        \end{align*}
    \end{small}where 
    \begin{small}
        \begin{align*}
            C_4(\xi,p)&=\exp\left\{3pL_2\hat{\sigma}^2\left(C_2(\xi)+L_1hC_3(\xi)\right)\right\}\\
            &= \exp\left\{3pL_2\hat{\sigma}^2|\xi|\right\}\exp\left\{3pL_2\hat{\sigma}^2\left(
                \int_{t}^{t_0+h}\alpha_sds+\frac{L_2}{2}h+L_1hC_3(\xi)
            \right)\right\}\\
            &:= \exp\left\{3pL_2\hat{\sigma}^2|\xi|\right\}\hat{\xi},
        \end{align*}
    \end{small}and 
    \begin{small}
        \begin{align*}
            \hat{\xi}=\exp\left\{3pL_2\hat{\sigma}^2\left(
                \int_{t}^{t_0+h}\alpha_sds+\frac{L_2}{2}h+L_1hC_3(\xi)
            \right)\right\}.
        \end{align*}
    \end{small}
    
    With H\"{o}lder's inequality under $G$-framework, Theorem \ref{Doob-inequality} and Remark \ref{Doob-exp}, we have
    \begin{small} 
        \begin{align*}
            &\widehat{\mathbb{E}}\left[\sup_{t\in[t_0,t_0+h]}\exp\left\{3pL_2\tilde{\sigma}^2\left|y_t^{(m)}\right|\right\}\right]\\
            &\leq \left(\widehat{\mathbb{E}}\left[\sup_{t\in[t_0,t_0+h]}\widehat{\mathbb{E}}_t
            \left[C_4^2(\xi,p)\right]\right]\right)^{\frac{1}{2}}
            \left(\widehat{\mathbb{E}}\left[\sup_{t\in[t_0,t_0+h]}\widehat{\mathbb{E}}_t
            \left[\exp\left\{6pL_2\hat{\sigma}^2L_1h
            \sup_{t\in[t_0,t_0+h]}|y^{(m)}_t|\right\}\right]\right]\right)^{\frac{1}{2}}\\
            &\qquad\times\widehat{\mathbb{E}}\left[\exp\left\{9pL_2\hat{\sigma}^2L_1h
            \frac{C}{c}\sup_{t\in[t_0,t_0+h]}|y^{(m)}_t|\right\}\right]\\
            &\leq \hat{A}_G\left(\widehat{\mathbb{E}}_t\left[C_4^4(\xi,p)\right]\right)^{\frac{1}{2}}
            \left(\widehat{\mathbb{E}}\left[\exp\left\{12pL_2\hat{\sigma}^2L_1h
            \sup_{t\in[t_0,t_0+h]}|y^{(m)}_t|\right\}\right]\right)^{\frac{1}{2}}\\
            &\qquad\times\widehat{\mathbb{E}}\left[\exp\left\{9pL_2\hat{\sigma}^2L_1h
            \frac{C}{c}\sup_{t\in[t_0,t_0+h]}|y^{(m)}_t|\right\}\right]\\
            &\leq \hat{A}_GC_5(\xi,p)\widehat{\mathbb{E}}\left[\exp\left\{
                \left(12p+18p\frac{C}{c}\right)L_2\hat{\sigma}^2L_1h\sup_{t\in[t_0,t_0+h]}|y^{(m)}_t|
            \right\}\right],
        \end{align*}
    \end{small}where $C_5(\xi,p)=\widehat{\mathbb{E}}[C_4^4(\xi,p)].$ Choose $h$ such that 
    $(4+6C/c)L_1h(\hat{\sigma}/\tilde{\sigma})^2<1$  
    then, for each $p\geq1,$
    \begin{small}
        \begin{align*}
            \widehat{\mathbb{E}}\left[\sup_{t\in[t_0,t_0+h]}\exp\left\{3pL_2\tilde{\sigma}^2
            \left|y_t^{(m)}\right|\right\}\right]\leq \left(\hat{A}_GC_5(\xi,p)
            \right)^{\frac{1}{1-(4+6C/c)L_1
            h(\hat{\sigma}/\tilde{\sigma})^2}}:=C_6(\xi,p)<\infty.
        \end{align*}
    \end{small}Then we obtain that  
    \begin{small}
        \begin{align*}
            \sup_{m}\widehat{\mathbb{E}}\left[\sup_{t\in[t_0,t_0+h]}
            \exp\left\{3pL_2\tilde{\sigma}^2\left|y_s^{(m)}\right|\right\}\right]&\leq
            \sup_{m}C_6(\xi,p) \leq C_7(\xi,p)<\infty
        \end{align*}
    \end{small}where $C_7(\xi,p)$ is a bounded constant with respect
    to $\xi,p$ and other constants given in the assumptions.

    Immediately, we have
    \begin{small}
        \begin{align*}%\label{estimate-Y-m}
            &\sup_{m}\widehat{\mathbb{E}}\left[\sup_{t\in[t_0,t_0+h]}
            \exp\left\{3pL_2\tilde{\sigma}^2\left|Y_t^{(m)}\right|\right\}\right]\notag\\
            &\leq \sup_{m}\widehat{\mathbb{E}}\left[\sup_{t\in[t_0,t_0+h]}
            \exp\left\{3pL_2\tilde{\sigma}^2\left(\left(1+3\frac{C}{c}\right)
            \left|y_t^{(m)}\right|+C_3(\xi)\right)\right\}\right]\\
            &\leq C_7(\xi,p)^{1+3\frac{C}{c}}\exp
            \left\{3pL_2\tilde{\sigma}^2C_3(\xi)\right\}<\infty.\notag
        \end{align*}
    \end{small}By the relationship of \eqref{Y-y-relationship}, we have $Z^{(m)}=z^{(m)},K^{(m)}=k^{(m)}$,
    according to the Lemma 3.7 in \cite{HTW22}, for any $p\geq1$, we obtain 
    \begin{small}
        \begin{align*}
            \sup_m \widehat{\mathbb{E}}\left[\left(\int_{t_0}^{t_0+h}\left|Z_t^{(m)}\right|^2dt\right)^n
            +\left|K_{t_0+h}^{(m)}\right|^n\right]< \infty, \ \forall n\geq1.
        \end{align*}
    \end{small}Moreover, similar to \eqref{estimate-bar-A}, we have
    \begin{small}
        \begin{align*}
            \sup_{m} \left|A_{t_0+h}^{(m)}\right|&=\sup_{m}\left|\bar{A}_{h}^{(m)}-\bar{A}_0^{(m)}\right|=
            \sup_{m}\left|\bar{A}_h^{(m)}-\bar{A}_h^3+\bar{A}_h^3\right|\\
            &\leq 3\frac{C}{c}\sup_{m}\widehat{\mathbb{E}}\left[\sup_{t\in [t_0,t_0+h]}|y^{(m)}_{t}|\right]
            +\frac{C}{c}\widehat{\mathbb{E}}\left[\left|\xi\right|\right]
            +\sup_{t\in [0,T]}\left|\bar{A}_t^3\right|< \infty.
        \end{align*}
    \end{small}
    
    The proof is complete.
\end{proof}
%\section*{Acknowledgement}
%We would like to thank Prof. Shige Peng for many helpful discussions.

\section*{Declarations}
We declare that there is no conflict of interest, and 
this manuscript is original, unpublished, and not under review elsewhere.

\end{document}